\RequirePackage{snapshot}
\documentclass[a4paper,11pt,twoside,leqno,english,fabfinal]{fabsmfart}

\usepackage{amsmath,fabalphabets,fabmacromath}
\usepackage{amssymb}
\usepackage{mathrsfs}
\usepackage{mathtools}
\usepackage{upgreek}
\let\oldcdot\cdot
\usepackage{breqn}
\let\cdot\oldcdot
\usepackage{graphicx}
\usepackage{stmaryrd}
\usepackage{gensymb}
\usepackage{wrapfig}
\usepackage{subfig}

\usepackage{caption}
\usepackage[backend=bibtex,style=alphabetic,giveninits=true,maxnames=99]{biblatex}
\usepackage{dsfont}

\usepackage{tikz}
\usetikzlibrary{arrows}
\usetikzlibrary{decorations.markings}
\usetikzlibrary{calc}
\usepackage{tkz-graph}

\usetikzlibrary{decorations.pathreplacing}
\usepackage{genyoungtabtikz}

\usepackage{hyperref}
\hypersetup{colorlinks=true, citecolor=ForestGreen, breaklinks, urlcolor=black, linkcolor=Black}

\usepackage[capitalise,nameinlink]{cleveref}

\theoremstyle{break}

\renewcommand{\theenumi}{\roman{enumi}}

\numberwithin{equation}{section}

\DeclareMathOperator{\BO}{\textit{O}}
\newcommand\BigO[1]{\BO{\hspace{-0.22em}\left(#1\right)}}
\DeclareMathOperator{\so}{\textit{o}}
\newcommand\smallo[1]{\so{\hspace{-0.15em}\left(#1\right)}}
\DeclareMathOperator{\Esp}{\mathbb{E}}
\newcommand\Expect[1]{\Esp{\hspace{-0.22em}\left[#1\right]}}
\DeclareMathOperator{\Pb}{\mathbb{P}}
\newcommand{\Prob}[1]{\Pb{\hspace{-0.22em}\left(#1\right)}}
\newcommand{\Va}{\operatorname{Var}}
\renewcommand\Var[1]{\Va{\hspace{-0.22em}\left(#1\right)}}
\renewcommand{\indic}{\mathds{1}}
\newcommand\Indic[1]{\indic{_{\{#1\}}}}

\DeclareCiteCommand{\cite}[\color{ForestGreen}\mkbibbrackets] 
  {\usebibmacro{prenote}}
  {\usebibmacro{cite}}
  {\multicitedelim}
  {\usebibmacro{postnote}}

\tikzset{->-/.style={decoration={
  markings,
  mark=at position #1 with {\arrow{>}}},postaction={decorate}}}
\tikzset{-<-/.style={decoration={
  markings,
  mark=at position #1 with {\arrow{<}}},postaction={decorate}}}

\title{Uniform random colored complexes}
\shorttitle{Uniform random colored complexes}
\author{A.~Carrance}

\begin{document}
\maketitle

\begin{fabsmfabstract}
  We present here random distributions on $(D+1)$-edge-colored, bipartite graphs with a fixed number of vertices $2p$. These graphs encode $D$-dimensional orientable colored complexes. We investigate the behavior of those graphs as $p \to \infty$. The techniques involved in this study also yield a Central Limit Theorem for the genus of a uniform map of order $p$, as $p \to \infty$.
\end{fabsmfabstract}

\section{Introduction}

For $D \geq 1$, we call \boldmath$(D+1)$\textbf{-colored graphs}, \unboldmath$(D+1)$-regular graphs, equipped with a proper $(D+1)$-coloring of their edges. $(D+1)$-colored graphs have been known from the 1970's, and the work of Pezzana {\cite{pez1,pez2}}, to be an encoding of piecewise-linear (PL) topological structures, that we will call \textbf{complexes} in a sense precised below. Among those structures are PL manifolds, which thus admit a combinatorial and graph-theoretical formulation. This formulation was further developped by Gagliardi and others (see {\cite{gagli}), leading notably to classification results for 3- and 4-dimensional PL (hence smooth) manifolds with a small number of cells {\cite{cristo1,cristo2}}. Additionally to these achievements in PL-topology, $(D+1)$-colored graphs have recently garnered interest from theoretical physicists, as they are at the heart of a new approach to quantum gravity, \textbf{colored tensor models} (see \cite{gurau-ryan,Gurau2016ab} for detailed reviews). As the quantized space-time described by colored tensor models is a PL-structure corresponding to a random distribution on $(D+1)$-colored graphs, this is an incentive to study such distributions.\\

  Our work can be related to other random models: first, we can compare it to Euclidean Dynamical Triangulations, that also have the purpose to define a quantum spacetime as the continuum limit of random triangulations in any dimension (see \cite{thorleifsson} for a detailed review). In this approach, the topology is fixed to be spherical, and the random distribution depends on a parameter that tunes a curvature constraint, while we do not put any constraint on the topology nor the curvature. Note also that we do not work with any triangulations, but with a specific type of complexes that we will define below. Moreover, these models have mostly been studied numerically, while we make a probabilistic and combinatorial investigation of ours.

 Likewise, we can relate this paper to works in random maps, as in dimension 2 our models can be seen as particular models of random maps. Furthermore, to look for a tentative continuum spacetime as the scaling limit of the discrete spacetime given by our model, we will consider the obtained complexes as metric spaces, just like the Brownian map is the scaling limit of several families of planar maps seen as metric spaces (see \cite{legall, miermont} for instance). While the most studied models of random maps have conditions on the \emph{face} degrees, with for instance random triangulations and quadrangulations, the graphs we consider for $D=2$ have 3-valent \emph{vertices}, so we are more naturally dealing with objects dual to the ones typically studied in the random maps literature. Moreover, most works on random maps deal with maps of a fixed genus (most often planar maps), or whose genus is assumed to grow linearly with their size (like in \cite{angel-chapuy-curien-ray}), whereas we do not fix the topology.

There is no fixed topology either in random simplicial complexes constructed as higher dimensional analogues of Erdös-Rényi graphs, such as \cite{costa-farber,Kahle2014aa}. Such models also share with the present work the fact that they are defined in any dimension. However, they are constructed quite differently, as their realizations are subcomplexes of a standard simplex (see the definitions below), whereas we build a complex by randomly gluing simplexes together.

This  gluing construction is very close to the various models of random gluings of polygons (see \cite{pipp-schleich, chmutov-pittel}), and we will indeed use similar techniques throughout this paper. Just like these random gluing models, our models can be seen as generalizations of the \textbf{configuration model}, which starts from a set of vertices with prescribed valences to form a random graph, by taking a uniform matching of the half-edges attached to these vertices. Some results on the configuration model will be of crucial use to prove some of our results.\\

Since the distributions arising the most naturally in colored tensor models are very involved from a mathematical point of view (notably, they are stongly non-uniform), we present here two simpler models on bipartite, vertex-labelled, $(D+1)$-colored graphs of fixed order $2p$. We focus on the limit $p \to \infty$, which is the first step towards the continuum limit for our tentative space-time. Our first model, which we analyze in \Cref{sect-unif}, is the obvious, uniform one: as we will see, the associated topological structure has a very singular behavior as $p \to \infty$. Our second model, which we call the \textbf{quartic model}, is described in \Cref{sect-quartic}. It has more physical foundation, and is richer topologically, but is still ill-fitted for the purpose of quantum gravity. Over the course of the study of this model, we also obtain a Central Limit Theorem for the genus of a uniform map of order $p$, as $p \to \infty$ (see \Cref{subsection-quartic-degree}). We then study in \Cref{sect-unif-decol} a generalization of the quartic model, that we call the class of \textbf{uniform-uncolored models}, which possess a similar asymptotical behavior.

Before going into technical details, let us fix some very general notations. We will say that an event $A$ occurs \textit{asymptotically almost surely} (a.a.s.), if $\Prob{A} \xrightarrow[p \to \infty]{} 1$. We will note $V(G)$ the vertex set of a graph $G$, $E(G)$ its edge set, and $k(G)$ its number of connected components (c.c.).

\subsection*{Colored graphs, bubbles and Gurau degree}

We now give a few necessary definitions about colored graphs.

\begin{defn}[Colored graphs]
Let $G$ be a $(D+1)$-regular graph, and $\zeta$ a function $E(G) \to C$, where $C$ is some set with cardinality $D+1$. The couple $(G,\zeta)$ is a \boldmath$(D+1)$\textbf{-colored graph}, if, for all \unboldmath$v \in V(G)$, for all $c \in C$, there is one and only one edge $e$ incident to $v$, such that $\zeta(e)=c$.\\
If there is no ambiguity, we will simply note $G$ for the colored graph. In that case, if not specified, $C$ will be the set of integers $\{0, 1, \dots, D \}$. 
\end{defn}

\begin{rem}
A very special class of $(D+1)$-colored graphs are those with only two vertices, and all edges joining the two vertices (see \Cref{melon}). We call those graphs, \textbf{melons}.
\end{rem}

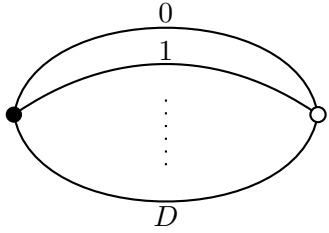
\begin{figure}[htp]
\centering
\begin{tikzpicture}
\draw [thick](-2,0) to [out=80,in=100](2,0);
\draw [thick](-2,0) to [out=35,in=145](2,0);
\draw [thick](-2,0) to [out=-80,in=-100](2,0);
\draw [fill] (-2,0) circle [radius=0.1];
\draw [fill=white,thick] (2,0) circle [radius=0.1];
\node at (0,1.35) {$0$};
\node at (0,0.85) {$1$};
\node at (0,-1.35) {$D$};
\draw [loosely dotted, thick] (0,0.2) -- (0,-0.7);
\end{tikzpicture}
\caption{Melons are the simplest example of ($D+1$)-colored graphs.}
\label{melon}
\end{figure}

\begin{defn}[Bubbles]\label{def-bubbles}
Let $G$ be a $(D+1)$-colored graph, and let $\{i_1, \dots, i_k\}$ be a subset of the color set $C$. Consider the graph $G_{\hat{\imath}_1, \dots, \hat{\imath}_k}$ obtained from $G$ by erasing all the $i_l$-colored edges, for $l=1, \dots, k$. The connected components of $G_{\hat{\imath}_1, \dots, \hat{\imath}_k}$ are called \boldmath$(D+1-k)$\textbf{-bubbles of $G$} (of colors \unboldmath$\{0, 1, \dots, D \} \backslash \{i_1,\dots, i_d\}$). We will write \boldmath$\mathcal{B}^{\hat{\imath}_1, \dots, \hat{\imath}_k}_{(\rho)}$ for a bubble of \unboldmath$G$ of colors $\{0, 1, \dots, D \} \backslash \{i_1,\dots, i_d\}$, with $\rho$ indexing the different bubbles of the same color. We will note \boldmath$b_k(G)$ the number of \unboldmath$k$-bubbles of $G$.\\
The vertices of $G$ are its 0-bubbles, its edges are the 1-bubbles, and its 2-bubbles (which are bicolored cycles) are called its \textbf{faces}.
\end{defn}

When dealing with bubbles, we will often write $\hat{\imath}_1 \dots \hat{\imath}_d$ for $\{0, 1, \dots, D \} \backslash \{i_1,\dots, i_d\}$, to simplify notations.

\begin{defn}[Embedding]
Let $G$ be a graph, and $S$ a Riemann surface. An \textbf{embedding of \boldmath$G$ into $S$} is a continuous and one-to-one function \unboldmath$i \colon G \to S$. We can then consider that $G$, as a topological space related to a 1-cellular complex, is included in $S$. The connected components of $S \setminus G$ are called the \textbf{regions} of the embedding. If all the regions are homeomorphic to an open disk, the embedding is said to be \textbf{2-cellular}.
\end{defn}

\begin{defn}[Regular embedding]
Let $G$ be a $(D+1)$-colored graph. A 2-cellular embedding of $G$ is said to be \textbf{regular} if there exists a $(D+1)$-cycle $\uptau \in \mathfrak{S}_{D+1}$, such that any region is bounded by a bicolored cycle, of colors $\{i,\uptau(i)\}$, for some $i \in C$. Any cycle $\uptau$ gives rise to a regular embedding.
\end{defn}

\begin{rem}
There is a two-to-two correspondence between regular embeddings of $G$ and $(D+1)$-cycles, as a cycle and its reverse are associated to the same two embeddings.
\end{rem}

If $G$ is $(D+1)$-colored and bipartite, all its regular embeddings are necessarily into orientable surfaces (see for instance \cite{gagli2}). This makes the following definition possible:

\begin{defn}[Degree]
The \textbf{Gurau degree} $\omega(G)$ of a $(D+1)$-colored graph $G$ is the sum of the genera of its regular embeddings:
\begin{equation*}
\omega(G) = \frac{1}{2} \sum_{\uptau} g_{\uptau}
\end{equation*}
where the sum runs over the $(D+1)$-cycles of $\kS_{D+1}$. The factor $1/2$ is a matter of convention, here we follow \cite{gurau-ryan}.
\end{defn}

The Gurau degree of $G$ can also be written in terms of its number of vertices, edges and faces:  

\begin{lemma}
\label{degree-faces}
\textnormal{\cite{gurau-rivasseau}}
For a connected bipartite $(D+1)$-colored graph $G$ with $2p$ vertices, one has:
\begin{equation*}
\omega(G) = \frac{(D-1)!}{2} \lbt\frac{D(D-1)}{2}p + D - b_2(G) \rbt.
\end{equation*}
\end{lemma}

\begin{proof}
$G$ has $\frac{D!}{2}$ distinct regular embeddings, and each face of $G$ corresponds to a region in $(D-1)!$ different regular embeddings.
\end{proof}

\subsection*{Trisps}
\label{sec-trisps}

With the bubbles of a $(D+1)$-colored graph $G$, we can build a $D$-dimensional triangulated space (trisp), which is a particular type of cellular complex with simplicial cells. For the sake of precision, let us fix some definitions and notations related to simplices that are needed for the definition of a trisp, for which we follow the conventions of \cite{kozlov}.

\begin{defn}[Simplices]
A \textbf{geometrical $\mathbf{n}$-simplex} $\sigma$ is the convex enveloppe of a set $A$ of $n+1$ affinely independent points in $\mathbb{R}^N$, for some $N \geq n$. The \textbf{dimension} of $\sigma$ is $|A|-1=n$. The convex enveloppes of the subsets of $A$ are called \textbf{sub-simplices} of $\sigma$, or its \textbf{faces}, and the points defining $\sigma$ are called its \textbf{vertices}. We will note $\sigma \subseteq \tau$ to signify that $\sigma$ is a face of $\tau$. A $d$-face of $\sigma$ is a face of dimension $d$, and the $\mathbf{d}$\textbf{-skeleton} of $\sigma$ is the set of its faces of dimension lower or equal to $d$.\\
We note $\langle x_1, \dots, x_n \rangle$ the simplex with vertices $x_1, \dots x_n$.\\
The \textbf{standard $\mathbf{n}$-simplex} is $\langle (1,0, \dots, 0), (0,1,0, \dots, 0), \dots, (0, \dots,0, 1) \rangle$, where the point coordinates are taken in the canonical basis of $\mathbb{R}^{n+1}$.
\end{defn}

Let us start from some sets $(S_i)_{i \in \mathbb{N}}$ of geometrical simplices, where $S_i$ contains $i$-simplices, seen as copies of the standard $i$-simplex. Then, for $m \leq n$, for each order-preserving injection $f \colon \{ 1, \dots, m+1\} \to \{ 1, \dots, n+1\}$, take a map $B_f \colon S_n \to S_m$, so that: 
\begin{enumerate}
\item \label{first} if $f, g$ are two such injections and are composable : $\{ 1, \dots, l+1\} \xrightarrow{g} \{ 1, \dots, m+1\} \xrightarrow{f} \{ 1,\dots, n+1\}$, then: $B_{f \circ g} = B_g \circ B_f$
\item \label{second} $B_{id_{\{1,\dots, n+1\}}} = id_{S_n}$.
\end{enumerate}
This abstract structure implicitly contains a topological space. Indeed, an order-preserving injection $f \colon \{ 1, \dots, m+1\} \to \{ 1, \dots, n+1\}$ induces a linear map $M_f \colon \mathbb{R}^{m+1} \to \mathbb{R}^{n+1}$, which sends the $k$-th vector of the canonical basis of $\mathbb{R}^{m+1}$ to the $f(k)$-th vector of the canonical basis of $\mathbb{R}^{n+1}$. This map can be restricted to a homeomorphism from the standard $m$-simplex to a certain $m$-sub-simplex of the standard $n$-simplex. For $\sigma \in S_n$, this homeomorphism therefore glues an $m$-sub-simplex of $\sigma$ to the standard $m$-simplex $B_f(\sigma) \in S_m$. The condition $B_{f \circ g} = B_g \circ B_f$ ensures that these gluings are coherent.

\begin{defn}[Trisp]
\label{trisp}
A complex defined in the above way is called a \textbf{triangulated space}, or \textbf{trisp}.
\end{defn}

Let $G$ be a $(D+1)$-colored graph. We now quickly map out the construction of a trisp $\Delta(G)$ with the bubbles of $G$. To each $(D+1-k)$-bubble $\mathcal{B}^{\hat{\imath}_1 \dots \hat{\imath}_k}_{(\rho)}$, we associate a $(k-1)$-simplex whose vertices are indexed by the missing colors $i_1, \dots, i_k$. This yields some set $S_k$, with an order on the vertices of each simplex. The gluing maps are defined in the following way: if $f \colon \{1,\dots, m+1\} \to \{1, \dots, n+1\}$ is an order-preserving injection, and if $\sigma \in S_n$ corresponds to the bubble $\mathcal{B}^{\hat{\imath}_1 \dots \hat{\imath}_{n+1}}_{(\rho)}$, $B_f$ sends it to the $m$-simplex corresponding to the $(D-m)$-bubble obtained from $\mathcal{B}^{\hat{\imath}_1 \dots \hat{\imath}_{n+1}}_{(\rho)}$ by adding the colors in $\set{i_{1},\dotsc,i_{n+1}}\setminus\{i_{f(1)}, \dots, i_{f(m+1)}\}$. By construction, the maps $B_f$ satisfy conditions (\ref{first}) and (\ref{second}) of \Cref{trisp}.\\

The trisp $\Delta(G)$ induced by a $(D+1)$-colored graph $G$ has severable notable properties:

\begin{thm}
\textnormal{\cite{gurau-2010}}\\
If $G$ is connected, the induced trisp $\Delta(G)$ is a simplicial pseudo-manifold of dimension $D$, \emph{i.e.}:
\begin{enumerate}
\item $S_i = \varnothing$ for $i > D$, and $S_D \neq \varnothing$
\item $\Delta(G)$ is \textbf{pure}, \emph{i.e.} any simplex is the face of a $D$-simplex
\item it is \textbf{strongly connected}, \emph{i.e.} any two $D$-simplices can be joined by a chain of $D$-simplices in which each pair of neighboring $D$-simplices has a common $(D-1)$-simplex
\item it is \textbf{non-branching}, \emph{i.e.} any $(D-1)$-simplex is a face of at most two $D$-simplices.
\end{enumerate}
Moreover, $\Delta(G)$ is a simplicial pseudo-manifold \textbf{without boundary}, \emph{i.e.} each of its $(D-1)$-simplices is actually a face of exactly two $D$-simplices.
\end{thm}

\begin{remsb}
\begin{itemize}
\item As the $(D+1-k)$-bubbles of $G$ correspond to the $(k-1)$-simplices of $\Delta(G)$, we sometimes say that $\Delta(G)$ is the \textbf{dual complex} of $G$.
\item If we cut open some edges of a bipartite $(D+1)$-colored graph $G$, the previous construction will yield a pseudomanifold with a boundary consisting of the $(D-1)$-simplices corresponding to the open half-edges.
\item As said before, the physical motivations of our work lead us to consider the complex $\Delta(G)$ as a metric space, with the structure given by its graph distance: the elements of this metric space are the vertices of $\Delta(G)$, \emph{i.e.} its 0-simplexes, and the distance between two points is the number of edges of the smallest path from one to the other in the 1-skeleton of $\Delta(G)$. The two first criteria to get a scaling limit will therefore be:
\begin{enumerate}
\item that the number of points of $\Delta(G)$ goes to infinity as $p \to \infty$
\item that the typical distance between two points of $\Delta(G)$ goes to infinity as $p \to \infty$.
\end{enumerate}
\end{itemize}
\end{remsb}

\subsection*{Permutations}
\label{sec-permutations}

Let us now note that the set of bipartite $(D+1)$-colored graphs with $2p$ labelled vertices (the positive and negative vertices being labelled independently from 1 to $p$ each), is in bijection with the $(D+1)$-tuples of permutations $(\alpha_0, \dots, \alpha_D) \in (\mathfrak{S}_p)^{D+1}$.\\

Indeed, given such a labelled graph, for each $i \in C$, define a permutation $\alpha_i$ by: $\alpha_i(k)=l$ if there is an $i$-colored edge linking the $k$-th negative vertex and the $l$-th positive one. Reciprocally, given such a tuple $(\alpha_0, \dots, \alpha_D)$, we can consider $p$ numbered negative vertices and $p$ numbered positive ones, and link them according to the permutations.\\

In the sequel, we will use the permutation formulation to define random bipartite vertex-labelled $(D+1)$-colored graphs.

\begin{rem}
The labellings of the vertices, combined with the coloring of the edges, yield a labelling of the half-edges, and with this formulation we count the possible pairings of the half-edges, \emph{i.e.} Wick contractions, in the language of quantum field theory. 
\end{rem}

\section{Uniform model}
\label{sect-unif}

We consider a $(D+1)$-tuple of random permutations $(\alpha_0, \dots, \alpha_D)$, all independent and uniform on $\mathfrak{S}_p$, for $D \geq 1$. These permutations induce a $(D+1)$-colored, bipartite random graph $U^{D}_p$, with $2p$ labelled vertices. It is clear that $U^{D}_p$ follows the uniform measure on this set of graphs.

\begin{rem}
We can note that this model is a ``colored version'' of the well-known configuration model \cite{Wormald1999aa}, that will appear more explicitly in \Cref{sect-quartic,sect-unif-decol}.
\end{rem}

\subsection{Connectedness}
\label{subsection-unif-conn}

We show, similarly to Pippenger and Schleich {\cite{pipp-schleich}}, that a.a.s. $U^{D}_p$ is connected, and more precisely:

\begin{thm}
\label{unif-connect}
Let $U^{D}_p$ be the random graph defined above, with $D \geq 2$. Then
\begin{equation*}
\Prob{U^{D}_p \text{ connected}} = 1 - \frac{1}{p^{D-1}} + \BigO{\frac{1}{p^{2(D-1)}}}.
\end{equation*}
\end{thm}

\begin{proof}
We first prove an upper bound on $\Prob{U^{D}_p \text{ not connected}}$. For $U^{D}_p$ to be not connected (n.c.), it must be decomposable into at least two closed subgraph, and, considering the smallest of these subgraphs, it must have a closed subgraph with at most $2\lfloor p/2 \rfloor$ vertices. Thus
\begin{align*}
\Prob{U^{D}_p \text{ n.c.}} \leq & \sum\limits_{1 \leq k \leq \lfloor \frac{p}{2} \rfloor} \Prob{\exists \text{ closed subgraph with $2k$ vertices}} \\
\leq & \sum\limits_{1 \leq k \leq \lfloor \frac{p}{2} \rfloor} F_k \text{ , where } F_k = \begin{pmatrix}
p \\
k
\end{pmatrix}^{1-D}.
\end{align*}
Since
\begin{equation*}
\frac{F_{k+1}}{F_k} = \frac{k+1}{p - k} \leq 1 \ \ \forall k=1, \dots, \Big\lfloor \frac{p}{2} \Big\rfloor - 1,
\end{equation*}
we get 
\begin{equation*}
\Prob{U^{D}_p \text{ n.c.}} \leq F_1 + F_2 + \left(\frac{p}{2} - 2 \right) F_3 = \frac{1}{p^{D-1}} + \left(\frac{2}{p(p-1)}\right)^{D-1} + \BigO{p^{-3D+4}},
\end{equation*}
and in particular
\begin{equation*}
\Prob{U^{D}_p \text{ connected}}\xrightarrow[p \to \infty]{} 1.
\end{equation*}

Now, to get a lower bound on $\Prob{U^{D}_p \text{ n.c.}}$, consider the probability of having exactly one closed melon: 
\begin{dmath*}
\Prob{U^{D}_p \text{ n.c.}} \geq \Prob{\exists !\text{ closed melon}}
= \Prob{\exists \text{ at least 1 closed melon}} - \Prob{\exists \text{ at least 2 closed melons}}
\geq \frac{1}{p^{D-1}} - \frac{1}{2}\frac{1}{(p(p-1))^{D-1}} \, .
\end{dmath*}
Thus $\Prob{U^{D}_p \text{ connected}} = 1 - \frac{1}{p^{D-1}} + \BigO{\frac{1}{p^{2(D-1)}}}$.
\end{proof}

Knowing that $U^{D}_p$ is connected a.a.s., the next step is to investigate its average number of connected components. We have the following result:

\begin{thm}
\label{unif-cc}
For $D\geq 2$, one has:
\begin{equation*}
\Expect{k(U^{D}_p)}= 1 + \BigO{\frac{1}{p^{D-1}}}.
\end{equation*}
\end{thm}

\begin{proof}
We derive upper bounds on $\Prob{k(U^{D}_p)=k}$, for $k \geq 2$. We already know that $\Prob{k(U^{D}_p)= 2} \leq \frac{1}{p^{D-1}} + \BigO{\frac{1}{p^{2(D-1)}}}$ from \Cref{unif-connect}. We then get an upper bound for $k=3$, decomposing the event that $U^{D}_p$ has 3 closed subgraphs: 
\begin{align*}
\Prob{k(U^{D}_p)=3} &\leq \Prob{k(U^{D}_p)\geq 3} = \Prob{U^{D}_p \text{ has } 3 \text{ closed proper subgraphs}}\\
&\leq \sum_{1 \leq k \leq \lfloor p/2 \rfloor} \Prob{U^{D}_p \text{ has a closed subgraph $U'$ with $2k$ vertices}}\\
&\ \ \ \cdot(\Prob{U' \text{ n.c.}} + \Prob{(U')^{c} \text{ n.c.})}\\
&\leq \sum_{1 \leq k \leq \lfloor p/2 \rfloor} \left(\frac{k!(p-k)!}{p!}\right)^{D-1}\\
&\ \ \ \cdot \left( \sum_{1 \leq l \leq \lfloor k/2 \rfloor}\left(\frac{l!(k-l)!}{k!}\right)^{D-1} + \sum_{1 \leq l \leq \lfloor (p-k)/2 \rfloor}\left(\frac{l!(p-k-l)!}{p-k!}\right)^{D-1}\right)\\
&\leq \frac{2}{p^{2(D-1)}} + \BigO{\frac{1}{p^{3(D-1)}}}.
\end{align*}
And, similarly, for $k \geq 4$:
\begin{align*}
\Prob{k(U^{D}_p) \geq 4 } \leq & \sum_{1 \leq k \leq \lfloor p/2 \rfloor} \Prob{U^{D}_p \text{ has a closed subgraph $U'$ with $2k$ vertices}} \\
& \cdot \left[\Prob{U' \text{ has } \geq 3 \text{ c.c.}} + \Prob{(U')^{c} \text{ has } \geq 3 \text{ c.c.}} + \Prob{U' \text{ n.c. and } (U')^c \text{ n.c.}}\right]\\
\leq & \BigO{\frac{1}{p^{3(D-1)}}}.
\end{align*}

Thus
\begin{align*}
\Expect{k(U^{D}_p)} & \leq 1 -\frac{1}{p^{D-1}} + \frac{2}{p^{D-1}} + \BigO{\frac{1}{p^{2(D-1)}}} + \left(\frac{p(p+1)}{2} - 6\right)\BigO{\frac{1}{p^{3(D-1)}}}\\
&\leq 1 + \BigO{\frac{1}{p^{D-1}}} \ \ \ \text{ for } D \geq 2. \\
\end{align*}
\end{proof}

\begin{rem}
For $D=1$, we have two uniform permutations $\alpha_0$ and $\alpha_1$. The number of connected components of $U^1_p$ is the number of cycles of $\alpha_0 \alpha_1^{-1}$, which is uniform too. Thus,  one has: $\Prob{U^{1}_p \text{ connected}} =\frac{1}{p}$, and, from well-known results on uniform permutations {\cite{arratia}}: $\Expect{k(U^1_p)}= \ln{p} + \BigO{1}$.
\end{rem}

For a given color $i \in \llbracket 0, D \rrbracket$, the graph $(U^{D}_p)_{\hat{\imath}}$, see \cref{def-bubbles}, has the same law as $U^{D-1}_p$ (up to a color renaming). This means that \Cref{unif-cc} can also be used for the number $b_D(U^{D}_p)$ of $D$-bubbles of $U^{D+1}_p$, simply by summing over all colors. This is of particular interest, as these $D$-bubbles correspond to the vertices of the complex dual to our colored graph. It is straightforward to derive the following result:  

\begin{cor}
For $D \geq 3$, one has:
\begin{equation*}
\Expect{b_D(U^{D}_p)}= D+1  + \BigO{\frac{1}{p^{D-2}}}.
\end{equation*}
\end{cor}

\begin{rem}
  This means that for $D \geq 3$, there is typically only one point of each color in the dual complex. This thwarts the hope of defining a continuous $D$-dimensional random space from this simple model by going through the same steps that yield the Brownian Sphere from uniform planar maps. Note that, as stated in the introduction, this uniform model differs from that of uniform planar maps, in that it does not fix the topology.

However, the essential difference might lie in the dimension: indeed, in Euclidean Dynamical Triangulations, Monte Carlo simulations show evidence of a so-called \textbf{crumpled phase} in dimension 3 and 4 even for spherical models, as long as the sampling of the triangulation does not depend too strongly on its curvature (see \cite{thorleifsson}), while for dimension 2 such a phase occurs only when there is no constraint on the topology.
\end{rem}

\subsection{Degree}
\label{subsection-unif-degree}

We now investigate the Gurau degree of $U^D_p$, for $D \geq 2$. According to \Cref{degree-faces}, this is equivalent to studying the number of faces of $U^D_p$, \textit{i.e.} the number of its bicolored cycles. This quantity can be expressed in terms of the permutations $(\alpha_0, \dots, \alpha_D)$: 
\begin{equation*}
b_2(U^D_p)= \sum_{0 \leq i < j \leq D} \mathcal{O}\left(\alpha_i \alpha_j^{-1}\right),
\end{equation*}
where $\mathcal{O}(\alpha)$ is the number of orbits (cycles) of $\alpha$.  The use of well-known results about uniform permutations gives us the following estimations for the average and variance of the number of faces:
\begin{propns}
\begin{align}
\label{unif-face-mean}
\Expect{b_2(U^D_p)} &= \frac{D(D+1)}{2}(\ln{p} + \gamma) + \smallo{1}\\
\label{unif-face-var}
\Var{b_2(U^D_p)}&= \frac{D(D+1)}{2}\ln{p} + \smallo{\ln{p}}.
\end{align}
\end{propns}
In terms of the Gurau degree, this means that:
\begin{align*}
\Expect{\omega(U^D_p)} &= \frac{(D-1)!}{2} \Big(\frac{D(D-1)}{2}p  + D - \frac{D(D+1)}{2}(\ln{p}+\gamma) \Big) + \smallo{1}\\
\Var{\omega(U^D_p)} &= \frac{(D-1)!}{2} \frac{D(D+1)}{2}\ln{p}+\smallo{\ln{p}} .
\end{align*}

\begin{proof}
If $\alpha$ is a uniform permutation of size $p$, then one has (see for instance {\cite{arratia}}):
\begin{align*}
\Expect{\mathcal{O}(\alpha)}&= \sum_{j=1}^{p}\frac{1}{j}=\ln{p}+\gamma+o(1) \text{, where $\gamma$ is the Euler constant}\\
\Var{\mathcal{O}(\alpha)}&= \sum_{j=1}^{p} \frac{j-1}{j^2}=\ln{p} + \gamma -\frac{\pi^2}{6} + o(1).
\end{align*}
\Cref{unif-face-mean} is obtained immediately from this. \Cref{unif-face-var} is obtained after simple calculations, once one notices that two permutations $\alpha_i \alpha_j^{-1}$ and $\alpha_k \alpha_l^{-1}$ are independent, as long as either $i \neq k$, or $j \neq l$.
\end{proof}

To get more precise information, we will now focus on the number of faces of a single regular embedding of $U^D_p$, instead of the total number of faces. We state our main result concerning this in \Cref{unif-jacket-normal-law} : the number of faces of any regular embedding $\mathcal{J}_p$ of $U^D_p$ has a normal limit when $p \to \infty$.\\

We can assume, without loss of generality, that $\mathcal{J}_p$ corresponds to the usual cyclic ordering of the colors $(0 \, 1 \, \dotsm \, D)$. Thus, we are interested in the law of:
\begin{equation*}
F = \sum_{0 \leq i \leq D} \mathcal{O}\left(\alpha_i \alpha_{i+1}^{-1}\right) =: \sum_{0 \leq i \leq D} \mathcal{O}_{i, i+1}
\end{equation*}
where, by convention, $\alpha_{D+1}=\alpha_0$.\\ 
To prove its normal asymptotical behaviour, we consider the distribution of the last permutation, $\alpha_D \alpha_0^{-1} =: \alpha_{D,0}$, conditionally to a given realization \{$\mathcal{C}_{i,i+1}$\} of the respective conjugacy classes $\pi(\alpha_{i,i+1})$ of the $D$ first permutations $\alpha_i \alpha_{i+1}^{-1} =: \alpha_{i,i+1}$. 
We note this distribution $P^{\{\mathcal{C}_{i,i+1}\}}_D$.\\
As $\alpha_{D,0}=\left(\prod_{0 \leq i \leq D-1}\alpha_{i,i+1}\right)^{-1}$, for some given $\{\mathcal{C}_{i,i+1}\}$, the parity of $\alpha_{D,0}$ is fixed. We will note $H:=\mathcal{A}_p$ or $\mathcal{A}_p^c$, according to this parity condition (where $\mathcal{A}_p$ is the alternating group of degree $p$), and $\mathcal{U}_H$ the uniform distribution on $H$. Let us assume that, for $i=0, \dots, D-1$, $\mathcal{C}_{i,i+1}$ has less than $\ln{p}$ fixed points and 2-cycles. We note this hypothesis $(*)$. As stated in \Cref{unif-asympt-indep}, under this assumption, $P^{\{\mathcal{C}_{i,i+1}\}}_D$ converges to $\mathcal{U}_H$ in total variation distance. We prove this using group representation techniques, similarly to Chmutov and Pittel in {\cite{chmutov-pittel}}.\\
Well-known results on the number of cycles of fixed length in a uniform permutation then allow us to deduce \Cref{unif-jacket-faces-law}, \textit{i.e.} that, up to a parity condition, the law $P_F$ of $F$ converges to a convolution product of the laws $P_{i,i+1}$ of the $\mathcal{O}_{i, i+1}$. Finally, we prove the asymptotic normality of $P_F$, using a well-known expression of the number of cycles of a uniform permutation as a sum of Bernoulli variables.\\

Let us now state these results more precisely:

\begin{thm}
\label{unif-jacket-normal-law}
Let $\mathcal{J}_p$ be a regular embedding of $U^D_p$, and $F_{\mathcal{J}_p}$ be the number of faces (\textnormal{i.e.} regions) of $\mathcal{J}_p$. Then the quantity $\frac{F_{\mathcal{J}_p} - \Expect{F_{\mathcal{J}_p}}}{\sqrt{\Var{F_{\mathcal{J}_p}}}}$ converges weakly to the standard normal distribution.  
\end{thm}

\begin{thm}
\label{unif-jacket-faces-law}
With the notations given above, one has:
\begin{equation*}
\lVert P_F - 2\cdot\Indic{(D+1)p-F \textnormal{ \scriptsize even}}P_{0,1} * P_{1,2} * \dotsm * P_{D,0} \rVert = \BigO{\frac{(\ln{p})^D}{p^{D-1}}}
\end{equation*}
where $\lVert \, \cdot \, \rVert$ is the total variation distance.
\end{thm}

\begin{prop}
\label{unif-asympt-indep}
Let us assume $(*)$. Then:
\begin{equation*}
\lVert P^{\{\mathcal{C}_{i,i+1}\}}_D - \mathcal{U}_H \rVert =\BigO{\frac{(\ln{p})^D}{p^{D-1}}}.
\end{equation*}
\end{prop}

Before starting the proof of \Cref{unif-asympt-indep}, it should be noted that, while its steps closely follow those of the proof of Theorem 2.2 in {\cite{chmutov-pittel}}, it employs a few stronger arguments, as we are here dealing with conjugacy classes with a logarithmic number of small cycles, whereas the permutations in {\cite{chmutov-pittel}} have no small cycles. We will detail those differences after the proof.

\begin{proof}
As proved in {\cite{chmutov-pittel}}, we get from the Cauchy-Schwartz inequality and the Plancherel theorem:
\begin{equation*}
\lVert P^{\{\mathcal{C}_{i,i+1}\}}_D - \mathcal{U}_H \rVert ^2 \leq \frac{1}{4} \sum_{\lambda \vdash p, \, \lambda \neq <p>, <1^p>} f^{\lambda} \text{tr}\left[\hat{P}^{\{\mathcal{C}_{i,i+1}\}}_D (\rho^{\lambda})\hat{P}^{\{\mathcal{C}_{i,i+1}\}}_D (\rho^{\lambda})^*\right]
\end{equation*}
where the sum is over the partitions $\lambda=(\lambda_1 \geq \lambda_2 \geq \dots)$ of $p$, $\rho^{\lambda}$ is the irreducible representation of $\mathfrak{S}_p$ associated to $\lambda$, $f^{\lambda}$ is the dimension of this representation, and, for a probability measure $P$ on $\mathfrak{S}_p$, $\hat{P}$ is the Fourier transform of $P$:
\begin{equation*}
\hat{P}(\rho) = \sum_{\sigma \in \mathfrak{S}_p} \rho(\sigma)P(\sigma)
\end{equation*}
for a representation $\rho$.

As the permutations $\alpha_{i,i+1}$, $i=0, \dots, D-1$, are all independent, the law of $\alpha_{D,0}$ writes as a convolution product (even conditionally to the realization of $\{ \mathcal{C}_{i,i+1}\}_{0 \leq i \leq D-1}$):
\begin{equation*}
P^{\{\mathcal{C}_{i,i+1}\}}_D=P_{\alpha_{D,D-1}} * P_{\alpha_{D-1,D-2}} * \dotsm * P_{\alpha_{1,0}}= \mathcal{U}_{\mathcal{C}_{D-1,D}} * \mathcal{U}_{\mathcal{C}_{D-2,D-1}} * \dotsm * \mathcal{U}_{\mathcal{C}_{0,D1}}.
\end{equation*}
Indeed, for $i=0,\dots,D-1$, $\alpha_{i,i+1}$ is uniform on $\mathcal{C}_{i,i+1}$, and so is its inverse $\alpha_{i+1,i}$. So, by applying the Fourier transform:
\begin{equation*}
\hat{P}^{\{\mathcal{C}_{i,i+1}\}}_D= \hat{\mathcal{U}}_{\mathcal{C}_{D-1,D}} \cdot \hat{\mathcal{U}}_{\mathcal{C}_{D-2,D-1}} \cdot \,  \dotsm \, \cdot \hat{\mathcal{U}}_{\mathcal{C}_{0,1}}.
\end{equation*}
Furthermore, as the $\mathcal{U}_{\mathcal{C}_{i,i+1}}$ are invariant under conjugacy, they are necessarily homothecies, by Schur's Lemma. We can therefore write:
\begin{equation*}
\hat{\mathcal{U}}_{\mathcal{C}_{i,i+1}}(\rho^{\lambda})=\frac{\chi^{\lambda}(\mathcal{C}_{i,i+1})}{f^{\lambda}}I_{f^{\lambda}}
\end{equation*}
where $\chi^{\lambda}$ is the character of $\rho^{\lambda}$. We derive from this:
\begin{equation*}
\lVert P^{\{\mathcal{C}_{i,i+1}\}}_D - \mathcal{U}_H \rVert ^2 \leq \frac{1}{4} \sum_{\lambda \vdash p, \, \lambda \neq <p>, <1^p>} \left(\frac{\prod_{0 \leq i \leq D-1}\chi^{\lambda}(\mathcal{C}_{i,i+1})}{(f^{\lambda})^{D-1}}\right)^2.
\end{equation*}

We now decompose this sum into two parts:
\begin{align*}
\lVert &P^{\{\mathcal{C}_{i,i+1}\}}_D - \mathcal{U}_H \rVert^2 \\
&\leq \frac{1}{4} \underbrace{\sum_{\substack{\lambda \neq <p>, <1^p> \\ \lambda_1, \lambda'_1 \leq p - 4}} \left(\frac{\prod_{0 \leq i \leq D-1}\chi^{\lambda}(\mathcal{C}_{i,i+1})}{(f^{\lambda})^{D-1}}\right)^2}_{\Sigma_1} + \underbrace{\sum_{\substack{\lambda \neq <p>, <1^p> \\ \lambda_1 \geq p - 3\text{ or } \lambda'_1 \geq p-3}} \left(\frac{\prod_{0 \leq i \leq D-1}\chi^{\lambda}(\mathcal{C}_{i,i+1})}{(f^{\lambda})^{D-1}}\right)^2}_{\Sigma_2}
\end{align*}
where $\lambda'$ is the dual partition of $\lambda$. We will bound $\Sigma_1$ and $\Sigma_2$ with different inequalities from Gamburd \cite{gamburd} and Larsen-Shalev {\cite{larsen-shalev}}. We start with $\Sigma_1$. Since we have assumed that $\mathcal{C}_{i,i+1}$ has at most $p^{\smallo{1}}$ cycles of size 1 or 2, from {\cite{larsen-shalev}}:
\begin{equation}
\label{larsen-shalev-bound}
\lvert \chi^{\lambda}\left(\mathcal{C}_{i,i+1}\right)\rvert \leq \left(f^{\lambda}\right)^{\frac{1}{3} + \smallo{1}} \text{ when } p \to \infty.
\end{equation}
Therefore:
\begin{equation*}
\forall \lambda \vdash p \ \ \left(\frac{\prod_{0 \leq i \leq D-1}\chi^{\lambda}(\mathcal{C}_{i,i+1})}{(f^{\lambda})^{D-1}}\right)^2 \leq \left(f^{\lambda}\right)^{-4D/3 +2 + \smallo{1}} \text{ when } p \to \infty.
\end{equation*}
And, from {\cite{gamburd}}:
\begin{equation*}
\forall t > 0 \sum_{\substack{\lambda \vdash p \\ \lambda_1, \lambda'_1 \leq p - m}}\left(f^{\lambda}\right)^{-t} = \BigO{p^{-mt}}.
\end{equation*}
Thus $\Sigma_1 = \BigO{p^{-4D + 6}}$.

For $\Sigma_2$, we use another inequality from {\cite{gamburd}}:
\begin{equation*}
f^{\lambda} \geq \begin{pmatrix}
p - a \\
a
\end{pmatrix} \text{ if } \lambda_1 = p - a> \frac{p}{2}
\end{equation*}
which also gives a bound on $f^{\lambda}$ when $\lambda'_1 \geq p- 3$, since $f^{\lambda}= f^{\lambda'}$. Now, similarly to {\cite{chmutov-pittel}}, we show that if $\lambda_1=p-a$, with $a= 1$, 2 or 3:
\begin{equation*}
\lvert \chi^{\lambda}(\mathcal{C}_i)\rvert = \BigO{(\ln{p})^a}.
\end{equation*}

\Yboxdim{1cm}
\newcommand\ylw{\Yfillcolour{yellow}}
\newcommand\wh{\Yfillcolour{white}}
\begin{figure}[htp]
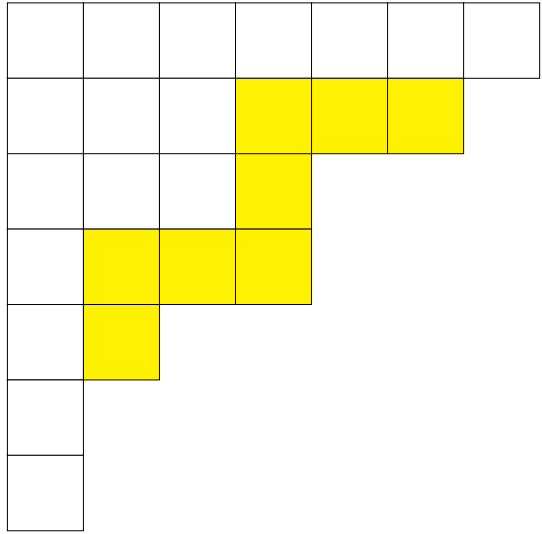

\centering
\gyoung(;;;;;;;,;;;!\ylw;;;,!\wh;;;!\ylw;,!\wh;!\ylw;;;,!\wh;!\ylw;,!\wh;,;)
\caption{Example of a rim hook.}
\label{rim-hook-fig}
\end{figure}
 Indeed, consider $\lambda$ with $\lambda_1 \geq p - 3$. Let $\zeta=(\zeta_1, \zeta_2, \dots)$ be an ordered sequence of the cycle lengths in $\mathcal{C}_{i,i+1}$ (with multiplicity). The Murnaghan-Nakayama rule implies:
\begin{equation}
\label{mur-nak}
\lvert \chi^{\lambda}(\mathcal{C}_i)\rvert \leq g^{\lambda}(\zeta)
\end{equation}
where $g^{\lambda}(\zeta)$ is the number of ways of emptying the Young diagram $Y(\lambda)$ associated to $\lambda$ by deleting rim hooks of successive sizes $(\zeta_1, \zeta_2, \dots)$. Here, by rim hooks of $Y$ we mean the contiguous border strips $R$ of $Y$ that can be removed from $Y$ leaving a proper subdiagram $Y \backslash R$ (see \Cref{rim-hook-fig}).

As $\lambda_1 = p - a \geq p - 3$, $Y(\lambda)$ is made of a first line of $p-a$ cells, and a small inferior sub-diagram $Z$, with $a$ cells. An ordered sequence of hook deletions emptying $Y(\lambda)$ according to $\zeta$ can be decomposed into two parts (see \Cref{rim-hook-deletions}): \\
- first, a sub-sequence of deletions that do no touch the $a$ first cells of the first line; \\
- then, a sub-sequence of deletions starting with one that touches the cell $(1,a)$.\\
The first sub-sequence is composed of horizontal deletions in the first line, and possibly some deletions in $Z$ too. Since at each step, the size of the deleted hook is fixed, the only freedom in this sub-sequence stems from the position of the deletions in $Z$. As $\mathcal{C}_{i,i+1}$ satisfies $(*)$, if $a=1$ there are at most $\BigO{\ln{p}}$ possible choices for the step at which the only possible deletion in $Z$ occurs, if $a=2$ there are at most $\BigO{(\ln{p})^2}$ possible choices (in the case where the two cells in $Z$ are deleted one by one), and likewise for $a=3$.

The second sub-sequence is composed of deletions in a sub-diagram containing at most $a^2$ cells, so the number of possibilities for this sub-sequence is a $\BigO{1}$. Therefore $\lvert \chi^{\lambda}(\mathcal{C}_i)\rvert \leq g^{\lambda}(\zeta) = \BigO{(\ln{p})^a}$.

\Yboxdim{1cm}
\newcommand\gr{\Yfillcolour{gray}}
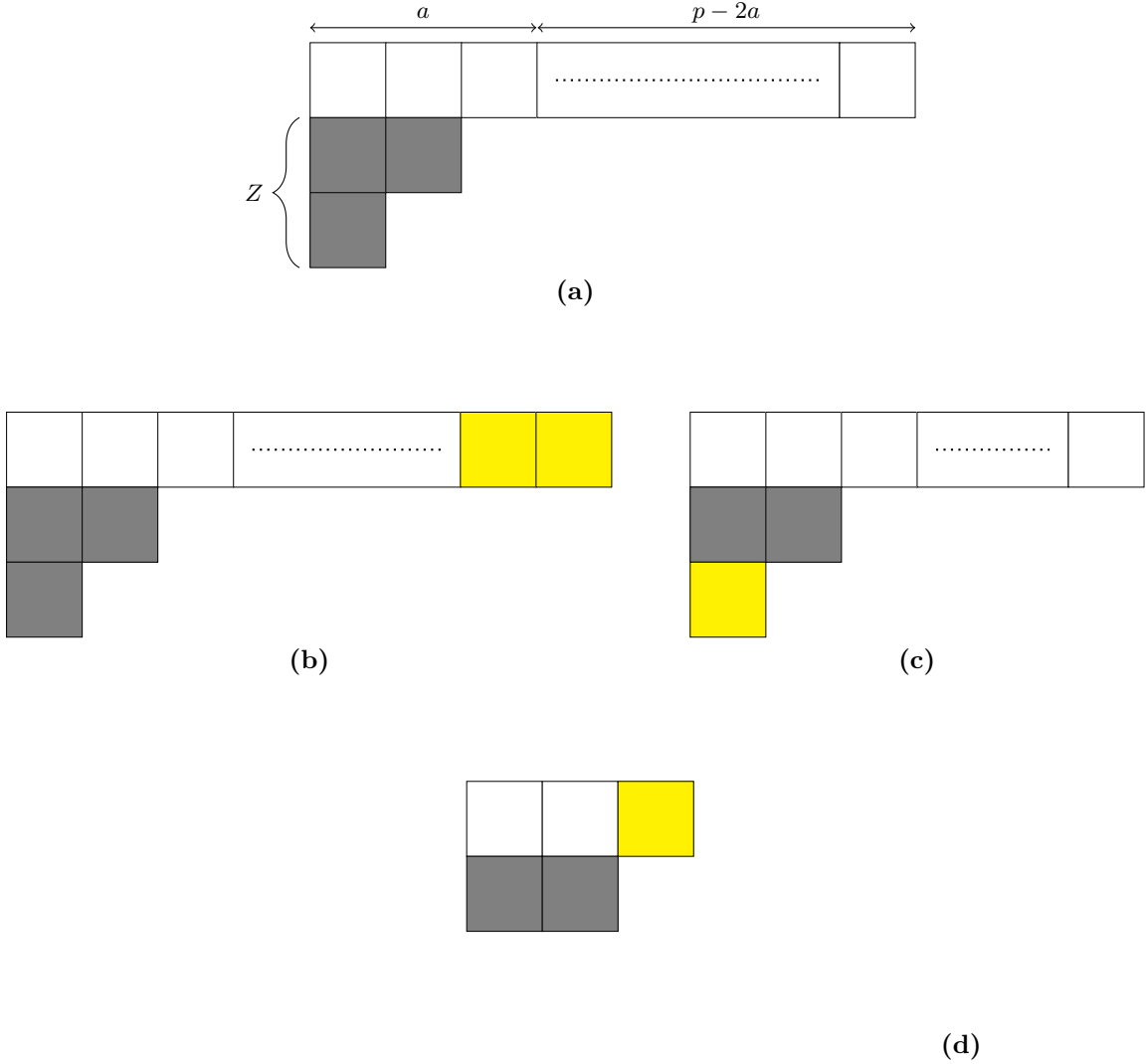
\begin{figure}[htp]
\centering
\subfloat[][]{
\begin{tikzpicture}
\tgyoung(0cm,3cm,;;;_4\hdts;,%
!\gr;;,%
;);
\draw [decorate,decoration={brace,amplitude=10pt},xshift=-4pt,yshift=0pt]
(0,1) -- (0,3) node [black,midway,xshift=-0.6cm] 
{\footnotesize $Z$};
\draw[<->](0,4.2)--(2.99,4.2) node [black,midway,yshift=0.2cm] 
{\footnotesize $a$};
\draw[<->](3.01,4.2)--(8,4.2) node [black,midway,yshift=0.2cm] 
{\footnotesize $p-2a$};
\end{tikzpicture}
\label{Y-start}}
\\
\vspace{1cm}
\subfloat[][]{
\begin{tikzpicture}
\tgyoung(0cm,3cm,;;;_3\hdts!\ylw;;,%
!\gr;;,%
;);
\end{tikzpicture}
\label{deletion-tail}}
\qquad
\subfloat[][]{
\begin{tikzpicture}
\tgyoung(0cm,3cm,;;;_2\hdts;,%
!\gr;;,%
!\ylw;);
\end{tikzpicture}
\label{deletion-bottom}}\\
\vspace{1cm}
\subfloat[][]{
\raisebox{1cm}{
\begin{tikzpicture}
\tgyoung(0cm,3cm,;;!\ylw;,%
!\gr;;,%
);
\end{tikzpicture}}
\label{Y-finish}}
\caption{The Young diagram $Y(\lambda)$ associated to $\lambda$ \protect\subref{Y-start}, the 2 possible types of deletions of the first subsequence \protect\subref{deletion-tail} and \protect\subref{deletion-bottom}, and the first deletion of the second subsequence \protect\subref{Y-finish}.}
\label{rim-hook-deletions}
\end{figure}
 
We get the same result when $\lambda'_1 \geq p - 3$, by considering the first column of $Y(\lambda)$. Hence
\begin{equation*}
\Sigma_2= \BigO{\frac{(\ln{p})^{2D}}{p^{2D-2}}},
\end{equation*}
which finally gives us $\lVert P^{(\mathcal{C}_1, \mathcal{C}_2)}_{\alpha_1 \alpha_2^{-1}} - \mathcal{U}_H \rVert = \BigO{\frac{(\ln{p})^D}{p^{D-1}}}$.
\end{proof}

As mentioned above, some steps of this proof use stronger arguments than in {\cite{chmutov-pittel}}. Indeed, to estimate $\Sigma_1$, we are in the strongest case of application of the bound \eqref{larsen-shalev-bound}, when there is a number $p^{\smallo{1}}$ of cycles under a given size, while the permutations in \cite{chmutov-pittel} have no cycles under a given size. Moreover, when we bound $\Sigma_2$ with a number of sequences of rim-hook deletions, we have to take into account the contribution to this number of rim-hooks of length 1 or 2, which do not appear in {\cite{chmutov-pittel}}, for the same reason.\\

We can now prove \Cref{unif-jacket-faces-law}.
\begin{proof}[of \Cref{unif-jacket-faces-law}]
For $0 \leq l \leq (D+1)p$, we want to estimate the probability
\begin{equation*}
\Prob{F=l} = \sum_{c_{0,1} + c_{1,2} + \dots + c_{D,0} =l} \Prob{\mathcal{O}_{0,1} = c_{0,1} , \mathcal{O}_{1,2} = c_{1,2}, \dots, \mathcal{O}_{D,0} = c_{D,0}}.
\end{equation*}
We decompose the event of having some given $\{c_{i,i+1}, 0 \leq i \leq D \}$ depending on the conjugacy classes $\pi(\alpha_{i,i+1})$ of the $\alpha_{i,i+1}$, for $0 \leq i \leq D-1$:
\begin{equation*}
\Prob{F=l} = \sum_{c_{0,1} + c_{1,2} + \dots + c_{D,0} =l}\ \ \sum_{\substack{\mathcal{C}_{i,i+1}\\ \mathcal{O(C}_{i,i+1})=c_{i,i+1} \\ \text{ for }i=0, \dots, D-1}}  \Prob{\mathcal{O}_{D,0} = c_{D,0} \big\lvert \{ \mathcal{C}_{i,i+1}\}} \prod_{0 \leq i \leq D-1}\Prob{\pi(\alpha_{i,i+1}) = \mathcal{C}_{i,i+1}}.
\end{equation*}
Now, we separate the conjugacy classes into those that satisfy $(*)$ and those that do not:
\begin{align*}
\Prob{F=l}  = & \sum_{c_{0,1} + c_{1,2} + \dots + c_{D,0} =l} \sum_{\substack{\mathcal{O(C}_{i,i+1})=c_{i,i+1} \\ \mathcal{C}_{i,i+1} \text{ satisfies } (*) \, \forall i}}  \Prob{\mathcal{O}_{D,0} = c_{D,0} \big\lvert \{ \mathcal{C}_{i,i+1}\}} \prod_{0 \leq i \leq D-1}\Prob{\pi(\alpha_{i,i+1}) = \mathcal{C}_{i,i+1}}\\
 + &\sum_{c_{0,1} + c_{1,2} + \dots + c_{D,0} =l}\sum_{\substack{\mathcal{O(C}_{i,i+1})=c_{i,i+1} \\ \exists i \, \, \mathcal{C}_{i,i+1} \text{ violating } (*)}}  \Prob{\mathcal{O}_{D,0} = c_{D,0} \big\lvert \{ \mathcal{C}_{i,i+1}\}} \prod_{0 \leq i \leq D-1}\Prob{\pi(\alpha_{i,i+1}) = \mathcal{C}_{i,i+1}}\\
=&: P_1 + P_2.
\end{align*}
From \Cref{unif-asympt-indep}, if all the $\mathcal{C}_{i,i+1}$ satisfy $(*)$, then:
\begin{align*}
&\Prob{\mathcal{O}\left(\alpha_{D,0}\right)=c_{D,0} \, \Big\lvert \, \mathcal{C}(\alpha_{i,i+1})=\mathcal{C}_{i,i+1}, 0 \leq i \leq D-1}\\
&=  \Prob{\mathcal{O}\left(\alpha_{D,0}\right)=c_{D,0}\, \Big\lvert \, \varepsilon(p-c_{D,0})=\varepsilon(\mathcal{C}_{0,1}) \dotsm \varepsilon(\mathcal{C}_{D-1,D})}  + \BigO{\frac{(\ln{p})^D}{p^{D-1}}}\\
&=  2 \cdot \Indic{\varepsilon(p-c_{D,0})=\varepsilon(\mathcal{C}_{0,1}) \dotsm \varepsilon(\mathcal{C}_{D-1,D})} \cdot \Prob{\mathcal{O}\left(\alpha_{D,0}\right)=c_{D,0}} + \BigO{\frac{(\ln{p})^D}{p^{D-1}}}
\end{align*}
where $\varepsilon(n)$ is the parity of $n$. Thus, for $P_1$, we know that the only dependency left in $c_{D,0}$ is a parity condition:
\begin{align*}
&\sum_{c_{0,1} + c_{1,2} + \dots + c_{D,0} =l} \ \  \sum_{\substack{\mathcal{O(C}_{i,i+1})=c_{i,i+1} \\ \mathcal{C}_{i,i+1} \text{ satisfies } (*) \, \forall i}}  \Prob{\mathcal{O}_{D,0} = c_{D,0} \big\lvert \{ \mathcal{C}_{i,i+1}\}} \prod_{0 \leq i \leq D-1}\Prob{\pi(\alpha_{i,i+1}) = \mathcal{C}_{i,i+1}} \\
= & \sum_{c_{0,1} + c_{1,2} + \dots + c_{D,0} =l}   2 \cdot \Indic{p - c_{D,0} \equiv Dp - \sum c_{i,i+1} \, [2]} \\
&\cdot \sum_{\substack{\mathcal{O(C}_{i,i+1})=c_{i,i+1} \\ \mathcal{C}_{i,i+1} \text{ satisfies } (*) \, \forall i}} \left(\Prob{\mathcal{O}_{D,0} = c_{D,0}} +\BigO{\frac{(\ln{p})^D}{p^{D-1}}}\right)\prod_{0 \leq i \leq D-1}\Prob{\pi(\alpha_{i,i+1}) = \mathcal{C}_{i,i+1}}\\
=&\sum_{c_{0,1} + c_{1,2} + \dots + c_{D,0} =l}   2 \cdot \Indic{p - c_{D,0} \equiv Dp - \sum c_{i,i+1} \, [2]}\\
& \cdot \left(\sum_{\substack{\mathcal{O(C}_{i,i+1})=c_{i,i+1} \\ \mathcal{C}_{i,i+1} \text{ satisfies } (*) \, \forall i}} \Prob{\mathcal{O}_{D,0} = c_{D,0}} \prod_{0 \leq i \leq D-1}\Prob{\pi(\alpha_{i,i+1}) = \mathcal{C}_{i,i+1}}\right)\\
& \ \ \ \ +\BigO{\frac{(\ln{p})^D}{p^{D-1}}} . 
\end{align*}

To control $P_2$, we use the fact that for a fixed $l$, the number of cycles of length $l$ in a uniform permutation $\alpha \in \mathfrak{S}_p$ asymptotically follows a Poisson law of parameter $1/l$ (see {\cite{arratia}}).\\
This implies that: $\Prob{\alpha\text{ has more than } \ln{p} \text{ cycles of size 1 and 2}} = \BigO{1/(\ln{p})!}$. Therefore:
\begin{align*}
P_2&=\Prob{F=l \text{ and } \exists i \in \{0,1, \dots, D-1\} \ \ \alpha_{i,i+1} \text{ violates } (*)} \\
&\leq \Prob{\exists i \in \{0,1, \dots, D-1\} \ \ \alpha_{i,i+1} \text{ violates } (*)}=\BigO{\frac{1}{(\ln{p})!}}
\end{align*}
and, similarly,
\begin{multline*}
\sum_{c_{0,1} + c_{1,2} + \dots + c_{D,0} =l} \sum_{\substack{\mathcal{O(C}_{i,i+1})=c_{i,i+1} \\ \exists i \, \, \mathcal{C}_{i,i+1} \text{ violating } (*)}} 2 \cdot \Indic{p - c_{D,0} \equiv Dp - \sum c_{i,i+1} \, [2]}\mathbb{P}(\mathcal{O}_{D,0}= c_{D,0})\cdot\\
\prod_{0 \leq i \leq D-1}\mathbb{P}(\pi(\alpha_{i,i+1}) = \mathcal{C}_{i,i+1})=\BigO{\frac{1}{(\ln{p})!}}.
\end{multline*}
Hence
\begin{multline*}
  \Prob{F=l}= \sum_{c_{0,1} + c_{1,2} + \dots + c_{D,0} =l} 2 \cdot \mathds{1}_{\{p - c_{D,0} \equiv Dp - \sum c_{i,i+1} \, [2]\}}\Prob{\mathcal{O}_{D,0}=c_{D,0}}\cdot\\
  \prod_{0 \leq i \leq D-1}\Prob{\mathcal{O}_{i,i+1} = c_{i,i+1}} + \BigO{\frac{\ln{p}^{D}}{p^{D-1}}}.
  \end{multline*}

Finally, notice that the parity condition on the $c_{i,i+1}$ is: ``$(D+1)p - \sum_{0 \leq i \leq D}c_{i,i+1}$ is even'', that is: ``$(D+1)p-l$ is even''. So
\begin{multline*}
\Prob{F=l}= 2 \cdot \Indic{(D+1)p-l \text{ \scriptsize even}}\sum_{\sum_{0 \leq i \leq D} c_{i,i+1} =l} \Prob{\mathcal{O}_{D,0}=c_{D,0}}\cdot\\
\prod_{0 \leq i \leq D-1}\Prob{\mathcal{O}_{i,i+1} = c_{i,i+1}} + \BigO{\frac{\ln{p}^{D}}{p^{D-1}}}.
\end{multline*}
\end{proof}

We now prove \Cref{unif-jacket-normal-law}:
\begin{proof}[of \Cref{unif-jacket-normal-law}]
We have obtained that:
\begin{equation*}
\lVert P_F - P_{\mathcal{F}} \rVert = \BigO{\frac{(\ln{p})^D}{p^{D-1}}},
\end{equation*}
where $P_{\mathcal{F}}=2\cdot\Indic{(D+1)p-\mathcal{F} \textnormal{ \scriptsize even}}P_{0,1} * P_{1,2} * \dotsm * P_{D,0}$, \emph{i.e.} $\mathcal{F}$ is, up to a parity condition, the sum of $D+1$ i.i.d. variables of law $P_{0,1}$. Therefore, $F$ converges (uniformly) in distribution to $\mathcal{F}$:
\begin{equation*}
\lvert \Phi_F - \Phi_{\mathcal{F}}\rvert = \BigO{\frac{(\ln{p})^D}{p^{D-1}}}
\end{equation*}
where $\Phi_X(a)=\Prob{X \leq a}$, \emph{i.e.} $\Phi_X$ is the cumulative distribution function of $X$.\\

We thus want to show that $\frac{\mathcal{F}- \Expect{\mathcal{F}}}{\sqrt{\Var{\mathcal{F}}}}$ converges weakly to the normal distribution. Now, for a uniform permutation $\sigma \in \mathfrak{S}_p$, $\mathcal{O}(\sigma)$ has the same distribution as a certain sum of independent Bernoulli variables:
\begin{equation*}
\Prob{\mathcal{O}(\sigma)=l}=\Prob{\sum_{1 \leq j \leq p} B_j =l}
\end{equation*}
 where the $B_j$ are independent Bernoulli variables, with $B_j$ of parameter $1/j$. Therefore, $\mathcal{F}$ has the same distribution as $\mathcal{S}\cdot \Indic{(D+1)p-\mathcal{S} \textnormal{ \scriptsize even}}$, with:
\begin{equation*}
\mathcal{S}=\sum_{1 \leq j \leq p }C_j
\end{equation*}
where the $C_j$ are independent binomial variables, with $C_j$ of parameters $(D+1, 1/j)$.\\

Applying the Lindeberg Central Limit Theorem to $\mathcal{S}$, we show that $\frac{\mathcal{S}-\Expect{\mathcal{S}}}{\sqrt{\Var{\mathcal{S}}}}$ converges in distribution to the standard normal law. Then, by applying the Local Limit Theorem (see {\cite[][Theorem 3.1]{mcdonald}}), we show that:
\begin{equation*}
\Prob{\mathcal{S}=l}=\frac{(1+\smallo{1})\exp{\left(\frac{-(l-\Expect{\mathcal{S}})^2}{2\Var{\mathcal{S}}}\right)}}{\sqrt{2\pi\Var{\mathcal{S}}}}
\end{equation*}
uniformly in $l$. Therefore, as
\begin{equation*}
\Prob{\mathcal{F}=l}=2 \cdot \Indic{(D+1)p-l \textnormal{ \scriptsize even}} \cdot \Prob{\mathcal{S}=l},
\end{equation*}
this local limit theorem holds for $\mathcal{F}$ as well.

This implies that $\frac{\mathcal{F}-\Expect{\mathcal{S}}}{\sqrt{\Var{\mathcal{S}}}}$ converges in distribution to the standard normal law, so this is also the case for $F$, and, as $\Expect{F}=\Expect{\mathcal{S}}$ and $\Var{F}=\Var{\mathcal{S}}$, we get the final result.
\end{proof}

\section{Quartic model}
\label{sect-quartic}

The uniform model presented in the previous section is easy to manipulate, but, as we have seen, it yields a relatively unsatisfying geometrical structure. Moreover, it is very far from the distributions arising in colored tensor models, as all of these yield an average number of faces linear with $p$ (see for instance \cite{bonzom-del-riv}), to be compared with the logarithmic behavior of the uniform case, which thus has little physical relevance. These two sources of dissatisfaction are incentives to consider a slightly more complicated model, that we call the \textbf{quartic model}, and whose structure is quite familiar to physicists working on colored tensor models \cite{Dartois2013aa,Delepouve2014ab,Ousmane-Samary2012ab,Rivasseau2016aa}. Note however that the quartic model from the physics literature has degree-dependent weights, whereas ours is uniform on the possible realizations.\\

More precisely, our quartic model can be defined as follows: we consider a bipartite $(D+1)$-colored graph, where the $D$-bubbles without color 0 are all quartic, \textit{i.e.} contain 4 vertices, put into two pairs linked by $D -1$ edges, with the two edges of the remaining color $c$ connecting the pairs (see \Cref{quartic-mel}). The $D$-bubbles without color 0, that we call \boldmath$\hat{0}$\textbf{-bubbles}, are the ``interaction vertices'' of the model from the point of view of quantum field theory, while the 0-edges are its ``propagators''. In each \unboldmath$\hat{0}$-bubble, we want every color $c \in \{1, 2, \dots, D\}$ to be drawn equiprobably as the distinguished color, the bubbles being all independent. We also want the edges of color $0$ to be sampled uniformly and independently of the rest.

\begin{figure}[htp]
\centering
\begin{tikzpicture}
\draw [thick](-1.5,0) to [out=60,in=120](1.5,0);
\draw [thick](-1.5,0) to [out=-60,in=-120](1.5,0);
\draw [thick](-1.5,-3) to [out=60,in=120](1.5,-3);
\draw [thick](-1.5,-3) to [out=-60,in=-120](1.5,-3);
\draw [thick] (-1.5,0) -- (-1.5,-3);
\draw [thick] (1.5,0) -- (1.5,-3);
\draw [fill] (-1.5,0) circle [radius=0.1];
\draw [fill=white,thick] (1.5,0) circle [radius=0.1];
\draw [fill=white,thick] (-1.5,-3) circle [radius=0.1];
\draw [fill] (1.5,-3) circle [radius=0.1];
\node at (-1.75,-1.5) {$c$};
\node at (1.7,-1.5) {$c$};
\node at (-1.9,0) {$2k$};
\node at (1.85,0) {$2k$};
\node at (-2.2,-3) {$2k-1$};
\node at (2.15,-3) {$2k-1$};
\draw [loosely dotted, thick] (0,0.3) -- (0,-0.3);
\draw [loosely dotted, thick] (0,-2.7) -- (0,-3.3);
\end{tikzpicture}
\caption{The structure of each $\hat{0}$-bubble of $Q^{D}_p$.}
\label{quartic-mel}
\end{figure}
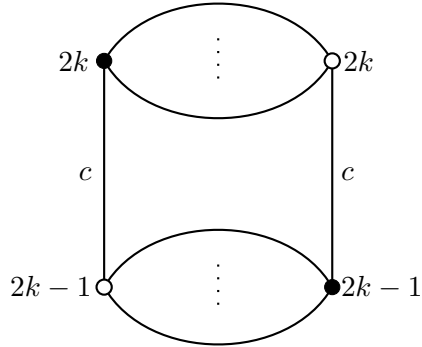

Thus, our random graph $Q^{D}_p$ ($(D+1)$-colored, with $p$ $\hat{0}$-bubbles, hence $4p$ vertices) can be described by a permutation $\alpha_0$ uniform on $\mathfrak{S}_{2p}$, independent from a set of $D$ permutations $\{\alpha_1, \alpha_2, \dots, \alpha_{D}\} \subseteq \mathfrak{S}_{2p}$, each being a product of transpositions of the form $(2k-1 \ \ \ 2k)$, with the constraints:
\begin{align*}
&\forall k \in \{1,2, \dotsc, p \}\  \exists ! i \in \{1, 2, \dots, D \}, \alpha_i(2k-1)=2k\\
& \forall k \in \{1,2, \dotsc, p \}, \forall i \in \{1, 2, \dots, D \},\ \Prob{\alpha_i(2k-1)=2k}=\frac{1}{D}.
\end{align*}

Following the structure of the previous section, we now present results on the connectedness and the degree of $Q^{D}_p$, as $p \to \infty$.

\subsection{Connectedness}
\label{subsect-quartic-cc}

We show that $Q^{D}_p$ is connected a.a.s., very similarly to the uniform case:
\begin{thm}
\label{quartic-connect}
For all $D \geq 2$, one has:
\begin{equation*}
\Prob{Q^{D}_p \text{ connected}}=1 - \frac{1}{2p-1} + \BigO{\frac{1}{p}}.
\end{equation*}
\end{thm} 
\begin{proof}
We use the same arguments as for \Cref{unif-connect}, with the slight difference that we consider the number of $\hat{0}$-bubbles contained in a subgraph, instead of the number of vertices. Thus, for the upper bound on $\Prob{Q^{D}_p \text{ n.c.}}$
\begin{align*}
\Prob{Q^{D}_p \text{ n.c.}} \leq & \sum\limits_{1 \leq k \leq \lfloor \frac{p}{2} \rfloor} \Prob{\exists \text{ closed subgraph with $k$ $\hat{0}$-bubbles}} \\
\leq & \sum\limits_{1 \leq k \leq \lfloor \frac{p}{2} \rfloor}  \begin{pmatrix} p \\ k \end{pmatrix} \frac{(2k)!(2(p-k))!}{(2p)!}= \frac{1}{2p-1} + \BigO{\frac{1}{p^2}},
\end{align*}
and, for the lower bound
\begin{align*}
\Prob{Q^{D}_p \text{ n.c.}} &\geq \Prob{\exists !\text{ isolated $\hat{0}$-bubble}} \\
&= \Prob{\exists \text{ at least 1 isolated bubble}} - \Prob{\exists \text{ at least 2 isolated bubbles}}\\
&\geq \frac{1}{2p-1} - \frac{1}{2(2p-1)(2p-3)} \, .
\end{align*}
\end{proof}

As was the case for the uniform model, the probability of $Q^{D}_p$ having $k$ connected components decreases fast enough with $k$ to get an asympotic expectation value of 1 for its number of connected components:
\begin{thm}
\label{quartic-cc}
For all $D \geq 2$, one has:
\begin{equation*}
\Expect{k(Q^{D}_p)}= 1 + \BigO{\frac{1}{p}}.
\end{equation*}
\end{thm}

\begin{proof}
Calculations similar to those of \Cref{subsection-unif-conn} give:
\begin{equation*}
\Prob{k(Q^{D}_p) \geq 3}=\BigO{\frac{1}{p^2}}, \ \ \Prob{k(Q^{D}_p) \geq 4}=\BigO{\frac{1}{p^3}}.
\end{equation*}
Hence
\begin{align*}
\Expect{k(Q^{D}_p)}&= 1 + \BigO{\frac{1}{p}} + \BigO{\frac{1}{p^2}} + \left(\frac{p(p+1)}{2} - 6\right)\BigO{\frac{1}{p^3}}\\
&= 1 + \BigO{\frac{1}{p}}.
\end{align*}
\end{proof}

If we now move on to the number of $D$-bubbles of $Q^D_p$, we get a picture which is very different from the uniform case. Indeed, since removing from $Q^D_p$ the edges of one given color does not yield $Q^{D-1}_p$, one cannot rely on \Cref{quartic-cc} to infer a result for $b_D(Q^D_p)$, and one must therefore tackle this problem with new tools.\\
Let us first deal with the case $D=2$, which is much simpler than the higher dimensions. As a matter of fact, for $D=2$ removing one color other than 0 (that is, either 1 or 2) yields $2p$ edges linked by a uniform permutation (see \Cref{i-bubbles-D2}), hence:
\begin{equation*}
\Expect{b_D(Q^2_p)}=p + 2(\ln{p} + \gamma) + \smallo{1},
\end{equation*}
using once again our knowledge of the number of cycles of uniform permutations.
\begin{figure}[htp]
\centering
\begin{tikzpicture}
\draw [thick,gray](-1,0) to [out=30,in=150] (1,0);
\draw [thick,gray](-1,-2) to [out=-30,in=-150] (1,-2);
\draw [thick] (-1,0) -- (-1,-2);
\draw [thick] (1,0) -- (1,-2);
\draw [thick](-1,0) -- (-1.25,0.25);
\draw [densely dotted,thick](-1.25,0.25) -- (-1.5,0.5);
\draw [thick](1,0) -- (1.25,0.25);
\draw [densely dotted,thick](1.25,0.25) -- (1.5,0.5);
\draw [thick](-1,-2) -- (-1.25,-2.25);
\draw [densely dotted,thick](-1.25,-2.25) -- (-1.5,-2.5);
\draw [thick](1,-2) -- (1.25,-2.25);
\draw [densely dotted,thick](1.25,-2.25) -- (1.5,-2.5);
\draw [fill] (-1,0) circle [radius=0.1];
\draw [fill=white,thick] (1,0) circle [radius=0.1];
\draw [fill=white,thick] (-1,-2) circle [radius=0.1];
\draw [fill] (1,-2) circle [radius=0.1];
\node at (0,0.5) {\textcolor{gray}{$i$}};
\node at (0,-2.5) {\textcolor{gray}{$i$}};
\node at (-1.4,0.15) {0};
\node at (1.4,0.15) {0};
\node at (-1.4,-2.15) {0};
\node at (1.4,-2.15) {0};
\draw [->, thick] (2.2,-1) -- (3.8,-1);
\draw [thick] (5,0) -- (5,-2);
\draw [thick] (7,0) -- (7,-2);
\draw [thick](5,0) -- (4.75,0.25);
\draw [densely dotted,thick](4.75,0.25) -- (4.5,0.5);
\draw [thick](7,0) -- (7.25,0.25);
\draw [densely dotted,thick](7.25,0.25) -- (7.5,0.5);
\draw [thick](5,-2) -- (4.75,-2.25);
\draw [densely dotted,thick](4.75,-2.25) -- (4.5,-2.5);
\draw [thick](7,-2) -- (7.25,-2.25);
\draw [densely dotted,thick](7.25,-2.25) -- (7.5,-2.5);
\draw [fill] (5,0) circle [radius=0.1];
\draw [fill=white,thick] (7,0) circle [radius=0.1];
\draw [fill=white,thick] (5,-2) circle [radius=0.1];
\draw [fill] (7,-2) circle [radius=0.1];
\node at (4.6,0.15) {0};
\node at (7.4,0.15) {0};
\node at (4.6,-2.15) {0};
\node at (7.4,-2.15) {0};
\end{tikzpicture}
\caption{Removing a color other than 0 in $Q^2_p$.}
\label{i-bubbles-D2}
\end{figure}
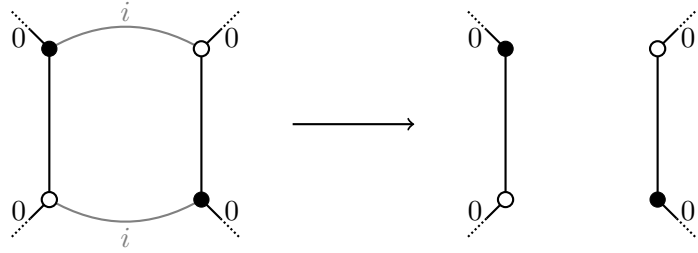\\

To answer our question for the general case $D \geq 3$, let us start with a quick analysis of the structure of $\left(Q^D_p\right)_{\hat{\imath}}$, for $i \in \{0,1, \dots, D\}$:
\begin{itemize}
\item if $i=0$, then we just have $p$ disconnected quartic bubbles; 
\item if $i \neq 0$, then $\left(Q^D_p\right)_{\hat{\imath}}$ is composed of $(D-1)$-melons (coming from the splitting of the $\hat{0}$-bubbles in which $i$ is distinguished) and quartic $(D-1)$-bubbles (coming from the $\hat{0}$-bubbles in which $i$ is not distinguished), completed with edges of color $0$.
\end{itemize}

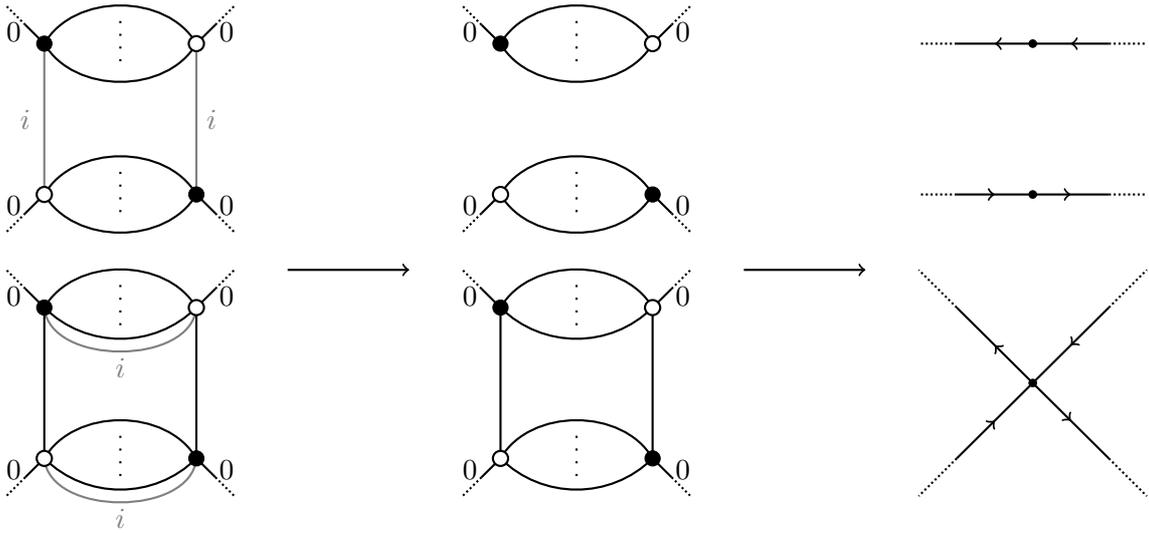
\begin{figure}[htp]
\centering
\begin{tikzpicture}
\draw [thick](-1,0) to [out=60,in=120](1,0);
\draw [thick](-1,0) to [out=-60,in=-120](1,0);
\draw [thick](-1,-2) to [out=60,in=120](1,-2);
\draw [thick](-1,-2) to [out=-60,in=-120](1,-2);
\draw [thick,gray] (-1,0) -- (-1,-2);
\draw [thick,gray] (1,0) -- (1,-2);
\draw [loosely dotted, thick] (0,0.3) -- (0,-0.3);
\draw [loosely dotted, thick] (0,-1.7) -- (0,-2.3);
\draw [thick](-1,0) -- (-1.25,0.25);
\draw [densely dotted,thick](-1.25,0.25) -- (-1.5,0.5);
\draw [thick](1,0) -- (1.25,0.25);
\draw [densely dotted,thick](1.25,0.25) -- (1.5,0.5);
\draw [thick](-1,-2) -- (-1.25,-2.25);
\draw [densely dotted,thick](-1.25,-2.25) -- (-1.5,-2.5);
\draw [thick](1,-2) -- (1.25,-2.25);
\draw [densely dotted,thick](1.25,-2.25) -- (1.5,-2.5);
\draw [fill] (-1,0) circle [radius=0.1];
\draw [fill=white,thick] (1,0) circle [radius=0.1];
\draw [fill=white,thick] (-1,-2) circle [radius=0.1];
\draw [fill] (1,-2) circle [radius=0.1];
\node at (-1.25,-1) {\textcolor{gray}{$i$}};
\node at (1.2,-1) {\textcolor{gray}{$i$}};
\node at (-1.4,0.15) {0};
\node at (1.4,0.15) {0};
\node at (-1.4,-2.15) {0};
\node at (1.4,-2.15) {0};
\draw [thick](-1,-3.5) to [out=60,in=120](1,-3.5);
\draw [thick](-1,-3.5) to [out=-45,in=-135](1,-3.5);
\draw [thick,gray](-1,-3.5) to [out=-85,in=-95](1,-3.5);
\draw [thick](-1,-5.5) to [out=60,in=120](1,-5.5);
\draw [thick](-1,-5.5) to [out=-45,in=-135](1,-5.5);
\draw [thick,gray](-1,-5.5) to [out=-85,in=-95](1,-5.5);
\draw [thick] (-1,-3.5) -- (-1,-5.5);
\draw [thick] (1,-3.5) -- (1,-5.5);
\draw [loosely dotted, thick] (0,-3.2) -- (0,-3.8);
\draw [loosely dotted, thick] (0,-5.2) -- (0,-5.8);
\draw [thick](-1,-3.5) -- (-1.25,-3.25);
\draw [densely dotted,thick](-1.25,-3.25) -- (-1.5,-3);
\draw [thick](1,-3.5) -- (1.25,-3.25);
\draw [densely dotted,thick](1.25,-3.25) -- (1.5,-3);
\draw [thick](-1,-5.5) -- (-1.25,-5.75);
\draw [densely dotted,thick](-1.25,-5.75) -- (-1.5,-6);
\draw [thick](1,-5.5) -- (1.25,-5.75);
\draw [densely dotted,thick](1.25,-5.75) -- (1.5,-6);
\draw [fill] (-1,-3.5) circle [radius=0.1];
\draw [fill=white,thick] (1,-3.5) circle [radius=0.1];
\draw [fill=white,thick] (-1,-5.5) circle [radius=0.1];
\draw [fill] (1,-5.5) circle [radius=0.1];
\node at (0,-4.3) {\textcolor{gray}{$i$}};
\node at (0,-6.3) {\textcolor{gray}{$i$}};
\node at (-1.4,-3.35) {0};
\node at (1.4,-3.35) {0};
\node at (-1.4,-5.65) {0};
\node at (1.4,-5.65) {0};
\draw [->, thick] (2.2,-3) -- (3.8,-3);
\draw [thick](5,0) to [out=60,in=120](7,0);
\draw [thick](5,0) to [out=-60,in=-120](7,0);
\draw [thick](5,-2) to [out=60,in=120](7,-2);
\draw [thick](5,-2) to [out=-60,in=-120](7,-2);
\draw [loosely dotted, thick] (6,0.3) -- (6,-0.3);
\draw [loosely dotted, thick] (6,-1.7) -- (6,-2.3);
\draw [thick](5,0) -- (4.75,0.25);
\draw [densely dotted,thick](4.75,0.25) -- (4.5,0.5);
\draw [thick](7,0) -- (7.25,0.25);
\draw [densely dotted,thick](7.25,0.25) -- (7.5,0.5);
\draw [thick](5,-2) -- (4.75,-2.25);
\draw [densely dotted,thick](4.75,-2.25) -- (4.5,-2.5);
\draw [thick](7,-2) -- (7.25,-2.25);
\draw [densely dotted,thick](7.25,-2.25) -- (7.5,-2.5);
\draw [fill] (5,0) circle [radius=0.1];
\draw [fill=white,thick] (7,0) circle [radius=0.1];
\draw [fill=white,thick] (5,-2) circle [radius=0.1];
\draw [fill] (7,-2) circle [radius=0.1];
\node at (4.6,0.15) {0};
\node at (7.4,0.15) {0};
\node at (4.6,-2.15) {0};
\node at (7.4,-2.15) {0};
\draw [thick](5,-3.5) to [out=60,in=120](7,-3.5);
\draw [thick](5,-3.5) to [out=-45,in=-135](7,-3.5);
\draw [thick](5,-5.5) to [out=60,in=120](7,-5.5);
\draw [thick](5,-5.5) to [out=-45,in=-135](7,-5.5);
\draw [thick] (5,-3.5) -- (5,-5.5);
\draw [thick] (7,-3.5) -- (7,-5.5);
\draw [loosely dotted, thick] (6,-3.2) -- (6,-3.8);
\draw [loosely dotted, thick] (6,-5.2) -- (6,-5.8);
\draw [thick](5,-3.5) -- (4.75,-3.25);
\draw [densely dotted,thick](4.75,-3.25) -- (4.5,-3);
\draw [thick](7,-3.5) -- (7.25,-3.25);
\draw [densely dotted,thick](7.25,-3.25) -- (7.5,-3);
\draw [thick](5,-5.5) -- (4.75,-5.75);
\draw [densely dotted,thick](4.75,-5.75) -- (4.5,-6);
\draw [thick](7,-5.5) -- (7.25,-5.75);
\draw [densely dotted,thick](7.25,-5.75) -- (7.5,-6);
\draw [fill] (5,-3.5) circle [radius=0.1];
\draw [fill=white,thick] (7,-3.5) circle [radius=0.1];
\draw [fill=white,thick] (5,-5.5) circle [radius=0.1];
\draw [fill] (7,-5.5) circle [radius=0.1];
\node at (4.6,-3.35) {0};
\node at (7.4,-3.35) {0};
\node at (4.6,-5.65) {0};
\node at (7.4,-5.65) {0};
\draw [->, thick] (8.2,-3) -- (9.8,-3);
\draw[->-=.5,thick] (12,0) -- (11,0);
\draw [densely dotted,thick](11,0) -- (10.5,0);
\draw[->-=.5,thick] (13,0) -- (12,0);
\draw [densely dotted,thick](13,0) -- (13.5,0);
\draw[->-=.5,thick] (11,-2) -- (12,-2);
\draw [densely dotted,thick](11,-2) -- (10.5,-2);
\draw[->-=.5,thick] (12,-2) -- (13,-2);
\draw [densely dotted,thick](13,-2) -- (13.5,-2);
\draw [fill] (12,0) circle [radius=0.05];
\draw [fill] (12,-2) circle [radius=0.05];
\draw[->-=.5,thick] (12,-4.5) -- (11,-3.5);
\draw [densely dotted,thick](11,-3.5) -- (10.5,-3);
\draw[->-=.5,thick] (13,-3.5) -- (12,-4.5);
\draw [densely dotted,thick](13,-3.5) -- (13.5,-3);
\draw[->-=.5,thick] (12,-4.5) -- (13,-5.5);
\draw [densely dotted,thick](13,-5.5) -- (13.5,-6);
\draw[->-=.5,thick] (11,-5.5) -- (12,-4.5);
\draw [densely dotted,thick](11,-5.5) -- (10.5,-6);
\draw [fill] (12,-4.5) circle [radius=0.05];
\end{tikzpicture}
\caption{Going from $Q^D_p$ to $\left(Q^D_p\right)_{\hat{\imath}}$, and from $\left(Q^D_p\right)_{\hat{\imath}}$ to $S^{D,i}_p$.}
\label{i-bubbles}
\end{figure}

For $i \neq 0$, as we are here only interested in the connected components of $\left(Q^D_p\right)_{\hat{\imath}}$, we can study a simpler (directed) graph $S^{D,i}_p$, which possesses vertices in lieu of quartic bubbles and melons, and whose edges correspond to the $0$-edges of $\left(Q^D_p\right)_{\hat{\imath}}$.\\
More precisely, the quartic bubbles of $\left(Q^D_p\right)_{\hat{\imath}}$ are represented in $S^{D,i}_p$ by vertices with in-degree and out-degree of 2, and its melons, by vertices with in-degree and out-degree of 1. We can decide for instance that each white (resp. black) vertex gives rise to an in- (resp. out-)half edge, and thus the half edges of $S^{D,i}_p$ unambiguously inherit a labelling. The uniformity of $\alpha_0$ is then translated into a uniform matching of the in- and out-half edges. This construction, illustrated in \Cref{i-bubbles}, clearly preserves the connected components of $\left(Q^D_p\right)_{\hat{\imath}}$.\\
Crucially, $S^{D,i}_p$ is a very special case of the (directed) configuration model, and general results by Cooper and Frieze {\cite{cooper-frieze}}, combined with the recent ones of Federico and van der Hofstad {\cite{f-vdh}}, allow us to obtain the following asymptotic results for the connected components of $\left(Q^D_p\right)_{\hat{\imath}}$: 

\begin{thm}
\label{quartic-i-bubbles}
For $D \geq 3$, for any $i \in \{1, 2, \dots, D\}$, $Q^D_p$ has a giant $\hat{\imath}$-bubble containing $4p - \BigO{\sqrt{p \ln{p}}}$ vertices. Moreover, the expectation value of the number of $\hat{\imath}$-bubbles of $Q^D_p$ is:
\begin{equation*}
\Expect{k(S^{D,i}_p)}= \ln{\left(\frac{D}{D - 1}\right)} +1 +o(1).
\end{equation*}
\end{thm}
We therefore get, for the total number of $D$-bubbles of $Q^D_p$:
\begin{cor}
\label{quartic-D-bubbles}
For $D \geq 3$:
\begin{equation*}
\Expect{b_D(Q^D_p)}= p + D\left(\ln{\left(\frac{D}{D - 1}\right)} +1\right) +o(1).
\end{equation*}
\end{cor}

As stated above, to prove \Cref{quartic-i-bubbles}, we will make use of results on the directed configuration model, that we define now. The directed configuration model describes a random digraph $D_n$ with $n$ vertices, among which, for $i,j \geq 0$, $l_{i,j}$ have in-degree $i$ and out-degree $j$. $D_n$ is obtained from these vertices by taking each matching of the in- and out-half-edges equiprobably. For $D_n$ to be well-defined,there must be an equal number of in- and out-half edges, \emph{i.e.}: $\sum_{i,j}il_{i,j} = \sum_{i,j}jl_{i,j}=:\theta \cdot n$. The average directed degree in $D_n$ is then $d:=\sum_{i,j}ij\frac{l_{i,j}}{\theta n}$.\\

$S^{D,i}_p$ is an example of a directed configuration model, with $q=2q_1 + q_2$ vertices, where $q_1$ is the number $\hat{0}$-bubbles that were split by removing $i$-edges, and $q_2= p - q_1$ is the number of bubbles that were not split. Now, by construction, $q_1$ is the sum of $p$ independent Bernoulli variables of parameter $1/D$, so, from the Hoeffding inequality:
\begin{equation*}
\Prob{\Big\lvert \frac{q_1}{p} - \frac{1}{D} \Big\rvert \geq \sqrt{\frac{\ln{p}}{p}}} \leq \frac{2}{p^2}.
\end{equation*}
Thus, a.a.s. $q_1 \sim p/D$, and in that case $l_{1,1}=\frac{2p}{D}(1 + \smallo{1})$, $l_{2,2}=\frac{(D-1)p}{D}(1 + \smallo{1})$, $\theta =\frac{2p}{q}=\frac{2D}{D+1}(1 + \smallo{1})$, and $d=\frac{2q_1 +4 q_2}{2p}=\left(\frac{2D-1}{D}\right)(1 + \smallo{1})$.
\\

To prove \Cref{quartic-i-bubbles}, we use a result proved in {\cite{cooper-frieze}}, which implies that $S^{D,i}_p$ a.a.s. contains one giant connected component of size of order $q$, as well as smaller cycles. Then, we will adapt a result from {\cite{f-vdh}} to know the number of those cycles.

\begin{thm}
\label{cooper-frieze-thm}
\textnormal{\cite{cooper-frieze}}\\
Let $D_n$ be a configuration model digraph defined by a sequence $(l_{i,j})$, satisfying:
\begin{enumerate}
\item the degrees in $D_n$ are bounded, \emph{i.e.} $\exists \, i_m,j_m$ such that $i > i_m$ or $j > j_m \Rightarrow l_{i,j} =0$
\item \label{theta0-d0} $\exists$ constants $\theta_0, d_0$ such that: $\theta=\theta_0(1+\smallo{1})$ and $d=d_0(1+\smallo{1})$, with $d_0 > 1$
\item $\forall \, i \ \ l_{i,0}=0$ and  $\forall \, j \ \ l_{0,j}=0$.
\end{enumerate}
Then a.a.s. the structure of $D_n$ is as follows:
\begin{enumerate}
\item There is a unique giant strongly connected component $S$ in $D_n$, of size $n - \BigO{\sqrt{n\ln{n}}}$.
\item There is a collection $C$ of vertex-disjoint directed cycles. The vertices of a cycle in $C$ are all of in-degree 1 or all of out-degree 1. Moreover, each cycle in $C$ is connected to $S$ by zero or more directed paths, all such paths having the same direction with respect to the given cycle.
\end{enumerate}
\end{thm}

To understand \Cref{cooper-frieze-thm}, it is important to note that it deals with \emph{strongly} connected components of $D_n$, \emph{i.e.} subgraphs of $D_n$ in which any two vertices can be joined by directed paths in both directions. Thus, the cycles in $C$ can be connected to $S$ by paths, but not strongly connected.\\

From our analysis of $S^{D,i}_p$, \Cref{cooper-frieze-thm} applies to our case, as only $l_{1,1}$ and $l_{2,2}$ are non-zero, $\theta \to \frac{2D}{D+1}$, and $d\to 2 - \frac{1}{D}$. Moreover, as all the vertices of $S^{D,i}_p$ have same in-degree and out-degree, its strongly connected components are the same as its merely connected ones. Therefore, $S^{D,i}_p$ is comprised of one giant connected component, together with a collection of small cycles. To obtain the number of connected components of $S^{D,i}_p$, we thus have to determine its number of cycles. To do so, we adapt a result from \cite{f-vdh} on the undirected configuration model, to the directed case:

\begin{prop}
\label{cycle-number-config-model}
Let $D_n$ be a configuration model digraph defined by a sequence $(l_{i,j})$, satisfying the hypothesis (\ref{theta0-d0}) from \Cref{cooper-frieze-thm}, and such that: $\frac{l_{1,1}}{n} \xrightarrow[p \to \infty]{} p_{1,1} \in [0,1]$. Let $C_k(n)$ be the number of cycles (composed of vertices of in- and out-degree 1) of length $k$ in $D_n$. Then:
\begin{equation*}
\left(C_k(n)\right)_{k \geq 1} \xrightarrow[]{(d)} \left(C_k\right)_{k \geq 1}
\end{equation*} 
where the $C_k$ are independent Poisson variables, of respective parameter $\frac{p_{1,1}^k}{k \theta_0 ^k}$.
\end{prop}

To prove \Cref{cycle-number-config-model}, we use the following \nameCref{multivariate-poisson}, similarly to the proof of Theorem 3.3 in {\cite{f-vdh}}:
\begin{lemma}
\label{multivariate-poisson}
\textnormal{\cite{vdh}} A sequence of vectors of non-negative integer-valued random variables $(X^{(n)}_1,X^{(n)}_2, \dots, X^{(n)}_k)_{n \geq 1}$ converges in distribution to a vector of independent Poisson random variables with parameters $(\lambda_1, \lambda_2, \dots, \lambda_k)$ when, for all possible choices of $(r_1, r_2, \dots, r_k) \in \mathbb{N}^k$:
\begin{equation*}
\lim_{n \to \infty}\Expect{(X_1^{(n)})_{r_1}(X_2^{(n)})_{r_2} \dotsm (X_k^{(n)})_{r_k}}=\prod_{i=1}^{k}\lambda_{i}^{r_{i}}
\end{equation*}
where $(X)_{r}=X(X-1) \dotsm (X -r +1)$.
\end{lemma}
\begin{proof}[of \Cref{cycle-number-config-model}]
From \Cref{multivariate-poisson}, it suffices to show that, for all $k \geq 1$ and all possible choices of $(r_1,r_2, \dots, r_k) \in \mathbb{N}^k$:
\begin{equation*}
\lim_{n \to \infty}\Expect{(C_1(n))_{r_1}(C_2(n))_{r_2} \dotsm (C_k(n))_{r_k}}=\left(\frac{p_{1,1}}{\theta_0}\right)^{r_1+2r_2 + \dotsm + kr_k} \frac{1}{2^{r_2} \dotsm k^{r_k}}.
\end{equation*}
We will proceed by induction on $k$. We first define $\mathscr{C}_k$ the set of candidates for cycles of length $k$ in $D_n$, \emph{i.e.} $\mathscr{C}_k=\{\{v_1, v_2, \dots, v_k\} \lvert v_i \text{ has in- and out-degree } 1\}$, so that:
\begin{equation*}
C_k(n)=\sum_{c \in \mathscr{C}_k} \Indic{c \text{ \scriptsize is in } D_n }.
\end{equation*}

This formulation highlights the fact that we are dealing with the factorial moments of sums of indicators, which implies that (see {\cite[][Section 2.1]{vdh}}):
\begin{multline}
\label{fact-moments-indic-sum}
\Expect{(C_1(n))_{r_1}(C_2(n))_{r_2} \dotsm (C_k(n))_{r_k}}= \\
\sum_{\substack{c_1(1), \dots, c_1(r_1) \in \mathscr{C}_1 \\ \text{distinct}}} \dotsm \sum_{\substack{c_k(1), \dots, c_k(r_k) \in \mathscr{C}_k\\ \text{distinct} }}\Prob{c_i(s) \text{ is in } D_n,\ \forall \, i=1, \dots, k,\ \forall s=1, \dots, r_k }.
\end{multline}

We note $W_k(\overrightarrow{r})$ the set of candidates for a collection of cycles like in \Cref{fact-moments-indic-sum}: $W_k(\overrightarrow{r}):=\{ \{c_1(1), \dots, c_1(r_1), \dots, c_k(1), \dots, c_k(r_k)\} \, \lvert \, c_i(j) \in \mathscr{C}_i, \, 1 \leq i \leq k, \, 1 \leq j \leq r_i \text{, the } c_i(s) \text{ are distinct}\}$, and for $w \in W_k(\overrightarrow{r})$ we note $\mathscr{E}(w)$ the event that all the cycles of $w$ are in $D_n$. Then
\begin{align*}
&\Expect{(C_1(n))_{r_1}(C_2(n))_{r_2} \dotsm (C_k(n))_{r_k}}\\
&=\sum_{w_k \in W_k(\overrightarrow{r})} \Prob{\mathscr{E}(w_k)}\\
&=\sum_{w_{k-1} \in W_{k-1}(\overrightarrow{r})} \Prob{\mathscr{E}(w_{k-1})}\underbrace{\sum_{\substack{c_1, \dots, c_{r_k} \in \mathscr{C}_k \\ \text{distinct}}} \Expect{\Indic{c_1 \text{ \scriptsize in } D_n } \dotsm \Indic{ c_{r_k} \text{ \scriptsize in } D_n }  \lvert \mathscr{E}(w_{k-1})}}_{S_k}.
\end{align*}

We now calculate $S_k$. It can be decomposed into the possible choices of cycles, that is, choices of collections of vertices of in- and out-degree of 1, multiplied by the probability of having those cycles in $D_n$. The choice of vertices is only constrained by the vertices already appearing in cycles of $w_k$, as those cannot be chosen. As for the probability of the cycles, they correspond to successive restrictions on the uniform permutation representing the matchings of the half-edges. Thus:
\begin{equation*}
S_k= \frac{(l_{1,1} - a_{1,1})!}{(l_{1,1} - a_{1,1}-kr_k)!(k!)^{r_k}} \cdot \frac{((k-1)!)^{r_k}(\theta n - a_{1,1}-k)!}{(\theta n - a_{1,1})!}
\end{equation*}
where $a_{1,1}$ is the number of vertices appearing in $w_{k-1}$. Therefore, for fixed $k$,
\begin{equation*}
S_k=\frac{l_{1,1}^{kr_k}}{k^{r_k}(\theta n)^{kr_k}}(1 + \smallo{1}) = \left(\frac{p_{1,1}^k}{k\theta_0^k}\right)^{r_k}(1 + \smallo{1}),
\end{equation*}
\emph{i.e.}
\begin{multline*}
\Expect{(C_1(n))_{r_1}(C_2(n))_{r_2} \dotsm (C_k(n))_{r_k}}=\\
\Expect{(C_1(n))_{r_1}(C_2(n))_{r_2} \dotsm (C_{k-1}(n))_{r_{k-1}}} \cdot\left(\frac{p_{1,1}^k}{k\theta_0^k}\right)^{r_k}(1 + \smallo{1}).
\end{multline*}
\end{proof}

From \Cref{cooper-frieze-thm} and \Cref{cycle-number-config-model}, we deduce that the number of connected components of $S^{D,i}_p$ has an expectation value of:
\begin{equation*}
\Expect{k(S^{D,i}_p)}= 1 + \sum_{k \geq 1}\Expect{C_k(q)}_{\Big\lvert \lvert q - \frac{D+1}{D}p\rvert = \BigO{\sqrt{p\ln{p}}}} + \BigO{p \cdot \frac{1}{p^2}} = 1 + \sum_{k \geq 1}\frac{1}{kD^k} + \smallo{1},
\end{equation*}
\emph{i.e.}
\begin{equation*}
\Expect{k(S^{D,i}_p)}= 1 + \ln{\left(\frac{D}{D-1}\right)} + \smallo{1}
\end{equation*}
which gives \Cref{quartic-i-bubbles}, and therefore \Cref{quartic-D-bubbles}.

\subsection{Geometry of the complex}

With a view towards scaling limits, it is encouraging that the number of points in the complex $\Delta(Q^D_p)$ grows linearly with $p$, as stated in \Cref{quartic-D-bubbles}. However, to make it possible to hope for a continuum scaling limit, the typical distances in $\Delta(Q^D_p)$ must grow to infinity with $p$. This is why we investigate the behavior of distances in this complex.\\

We show that, unfortunately, the average distance between two points of the complex is bounded, and more precisely, as stated in \Cref{quartic-avg-dist}, equal to 2. This can be predicted by considering the  ``typical landscape'' of the complex. Indeed, there always are $p$ $0$-points, while the number of $i$-points, for $i \neq 0$, grows sublinearly with $p$, as a consequence of \Cref{cooper-frieze-thm}. Thus, a uniform point is a.a.s. of color 0. Moreover, \Cref{cooper-frieze-thm} also implies that, for $i \neq 0$, there is a ``hub'' $i$-point corresponding to the giant component of $S^{D,i}_p$, which is linked to most of the $0$-points, and much more isolated $i$-points, corresponding to the small cycles of $S^{D,i}_p$. Hence, typically two independent and uniform points of  $\Delta(Q^D_p)$ will be of color $0$, and will have the ``hub'' $i$-points, for $i \neq 0$, as common neighbors. We formalize this heuristic in the proof of the following theorem:

\begin{thm}
\label{quartic-avg-dist}
Let $u,v$ be two vertices of $\Delta(Q^D_p)$ chosen uniformly at random and independently. Then, a.a.s.:
\begin{equation*}
d(u,v)=2.
\end{equation*}
\end{thm}

\begin{proof}
To prove this result, we show that a.a.s. $u$ and $v$ are points of color 0, and have a common neighbor of color $i \neq 0$. The fact that they are 0-colored a.a.s. is a simple consequence of \Cref{cooper-frieze-thm}, and more precisely the fact that a.a.s. the giant component of $S^{D,i}_p$ has size $q - \BigO{\sqrt{q\ln{q}}}$ (where $q$ is the number of vertices of $S^{D,i}_p$), so that a.a.s. there are less than $\BigO{\sqrt{p\ln{p}}}$ $\hat{\imath}$-bubbles for $i \neq 0$, while there are $p$ $\hat{0}$-bubbles.

It now remains to show that $u$ and $v$ have a common neighbor. To do so, let us consider the $D$-bubbles $\mathcal{B}_u, \mathcal{B}_v$ of $Q^D_p$ they respectively correspond to. For any color $i \neq 0$, $\mathcal{B}_u$ corresponds to either 1 or 2 vertices in $S^{D,i}_p$, depending on whether $i$ is the distinguished color in $\mathcal{B}_u$ or not. If $\mathcal{B}_u$ corresponds to a single vertex $u \in S^{D,i}_p$, then conditionally to the existence of a giant component like in \Cref{cooper-frieze-thm}, $u$ is necessarily in that giant component, as it is a quartic vertex. Now, conditionally to the fact that $\mathcal{B}_u$ corresponds to 2 vertices $w, x \in S^{D,i}_p$, both $w$ and $x$ are uniform on the quadratic vertices of $S^{D,i}_p$, and are thus a.a.s. in the giant component of $S^{D,i}_p$, from \Cref{cooper-frieze-thm}. Therefore, a.a.s. $\mathcal{B}_u$ and the $\hat{\imath}$-bubble $\mathcal{B}_{g,i}$ corresponding to the giant component of $S^{D,i}_p$ have an $(\hat{\imath},\hat{0})$-bubble in common. In terms of simplices of $\Delta(Q^D_p)$, this means that a.a.s., a 1-simplex of the complex links $u$ to the vertex $t_{g,i}$ corresponding to $\mathcal{B}_{g,i}$, \emph{i.e.} that these two vertices are nearest neighbors. Following the same reasoning for $v$, a.a.s. it is also a direct neighbor of $t_{g,i}$, so that $u$ and $v$ are at distance 2 in $\Delta(Q^D_p)$.
\end{proof}

\subsection{Degree}
\label{subsection-quartic-degree}

We now look into the number of faces of $Q^D_p$. Like for the uniform model, we present results on the expectation value and the variance of the total number of faces, as well as one on the asymptotic behavior of the number of faces of a given jacket. As we will see, upon proving these results we also obtain the asymptotic behavior of the genus of a uniform map of size $p$, as $p \to \infty$ (see \Cref{rib-graph-jack}).
\begin{thmns}[Number of faces of $Q_{p}^{D}$]
\label{quartic-face-moments}
\begin{align}
\label{quartic-face-mean} \Expect{[b_2(Q^D_p)} & = (D-1)^2p + D(\ln{(2p)}+\gamma) +\smallo{1}\\
\label{quartic-face-var} \Var{b_2(Q^D_p)} & =  \BigO{(\ln{p})^3}.
\end{align}
\end{thmns}
We also get a normal limit for the number of faces of one jacket:
\begin{thm}
\label{quartic-jacket-normal-law}
Let $\mathcal{J}_p$ be a regular embedding of $Q^D_p$, and $F_{\mathcal{J}_p}$ be the number of faces of $\mathcal{J}_p$. Then the quantity $ \frac{F_{\mathcal{J}_p} - \Expect{F_{\mathcal{J}_p}}}{\sqrt{\Var{F_{\mathcal{J}_p}}}}$ converges weakly to the standard normal distribution.  
\end{thm}

To prove both \Cref{quartic-face-moments,quartic-jacket-normal-law}, we will need to determine the asymptotical behavior of the joint law of $\mathcal{O}(\alpha_0 \alpha_i^{-1}), \mathcal{O}(\alpha_j \alpha_i^{-1}), \mathcal{O}(\alpha_j \alpha_0^{-1})$, for $1 \leq i < j \leq D$. By conjugating all permutations by a uniform permutation $\beta \in \mathfrak{S}_{2p}$, independent from the rest, this amounts to answering the same question for $\mathcal{O}(\varphi), \mathcal{O}(\alpha), \mathcal{O}(\alpha \varphi^{-1})$, for $\alpha, \varphi$ two independent permutations in $\mathfrak{S}_{2p}$, with $\varphi$ uniform, and $\alpha$ an involution with $n=2b$ fixed points, where $b$ is the sum of $p$ independent Bernoulli variables of parameter $\frac{D-2}{D}$. This slight change of setting leads to considering the following result:

\begin{prop}
\label{quartic-simpl-asympt}
Let $p$ be a positive integer, and $0 \leq b \leq p$ fixed. Let $\alpha$ be uniform on the conjugacy class of involutions on $\{1, \dots, 2p \}$ with $2b$ fixed points, and let $\varphi$ be a uniform permutation on $\{1, \dots, 2p \}$, independent from $\alpha$. Then $(\mathcal{O}(\varphi)+\mathcal{O}(\alpha \varphi^{-1}))$ has the same law as $(\mathcal{O}(\psi) + \mathcal{O}(\delta \psi^{-1}))$, where $\delta$ is uniform on the fixed-point-free involutions on $\{1, \dots, 2(p-b)\}$, and $\psi$ is uniform on $\mathfrak{S}_{2(p-b)}$, and independent from $\delta$.
\end{prop}

\begin{proof}
The key idea of the proof is that $\alpha, \varphi, \alpha \varphi^{-1}$ respectively describe the (half-)edges, faces and vertices, of an (half-edge) labelled ribbon graph $G$ with $(p-b)$ edges and $2b$ unmatched half-edges. By erasing the unmatched half-edges of $G$, we obtain a ribbon graph $H$ with $(p-b)$ edges.

As $\alpha$ and $\varphi$ are independent and uniform on their respective sets of possible realizations, $G$ is uniform on the set of labelled ribbon graphs with $(p-b)$ edges and $2b$ unmatched half-edges. Therefore, if we relabel the half-edges of $H$, by keeping their order, so that the labels are in $\{1, 2, \dots, 2(p-b)\}$, we get a ribbon graph $H'$, uniform on the labelled ribbon graphs with $(p-b)$ edges. $H'$ can thus be identified to a triple $(\delta, \psi, \delta \psi^{-1})$ satisfying the hypotheses mentioned above. As the transformation from $G$ to $H'$ preserves the Euler characteristic, we have
\begin{align*}
\chi(G) &= \mathcal{O}(\varphi) - \left( \mathcal{O}(\alpha) - 2b \right) + \mathcal{O}(\alpha \varphi^{-1})\\
&= \chi(H')=\mathcal{O}(\psi) - \mathcal{O}(\delta) + \mathcal{O}(\delta \psi^{-1}),
\end{align*}
that is
\begin{equation*}
\mathcal{O}(\varphi) + \mathcal{O}(\alpha \varphi^{-1}) = \mathcal{O}(\psi) + \mathcal{O}(\delta \psi^{-1}).
\end{equation*}
\end{proof}

Let us go back to the permutations that define $Q^D_p$. We deduce from \Cref{quartic-simpl-asympt} an intermediate result that will be necessary to prove \Cref{quartic-face-moments}:
\begin{prop}
\label{quartic-asympt-bound}
Let $\mathcal{O}_{i,j}$ be the number of cycles of $\alpha_i \alpha_j^{-1}$. Then:
\begin{equation*}
\Expect{\mathcal{O}_{0,1}\mathcal{O}_{D,0}} =\BigO{(\ln{p})^3}.
\end{equation*}
\end{prop}

\begin{proof}
We first write:
\begin{equation*}
\Expect{\mathcal{O}_{0,1}\mathcal{O}_{D,0}}=\frac{1}{2}\left(\Expect{\left(\mathcal{O}_{0,1}+ \mathcal{O}_{D,0}\right)^2} - \Expect{\mathcal{O}_{0,1}^2} + \Expect{\mathcal{O}_{D,0}^2}\right).
\end{equation*}
Then, from \Cref{quartic-simpl-asympt}, when the number $2b$ of fixed points of $\alpha_D \alpha_1 ^{-1}$ is fixed, we have:
\begin{equation*}
\Expect{\left(\mathcal{O}_{0,1}+ \mathcal{O}_{D,0}\right)^2 \, \lvert \, \alpha_{D,1} \text{ has } 2b \text{ fixed points }} = \Expect{\left(\mathcal{O}(\psi) + \mathcal{O}(\delta \psi^{-1})\right)^2}
\end{equation*}
with the notations of \Cref{quartic-simpl-asympt}. Now, with the same techniques as in \Cref{subsection-unif-degree}, we prove that, conditionally to the conjugacy class $\mathcal{C}_{\psi}$ of $\psi$, $\delta \psi^{-1}$ is asymptotically uniform on $H=\mathcal{A}_{2(p-b)}$ (resp.  $H=(\mathcal{A}_{2(p-b)})^{c}$) if $\psi$ and $\delta$ are of the same parity (resp. of opposite parities). More precisely,
\begin{equation*}
\lVert P_{\delta \psi^{-1}} - \mathcal{U}_H \rVert = \BigO{\frac{(\ln{(p-b)})^2}{p-b}}.
\end{equation*}
Thus
\begin{align*}
\Expect{\mathcal{O}(\psi)\mathcal{O}(\delta \psi^{-1})}&=\sum_{1 \leq c_1, c_2 \leq 2(p-b)} c_1 c_2 \Prob{\mathcal{O}(\psi)=c_1, \mathcal{O}(\delta \psi^{-1})=c_2}\\
&=\sum_{1 \leq c_1 \leq 2(p-b)} c_1 \Prob{\mathcal{O}(\psi)=c_1}
  \Big[\BigO{(\ln{(p-b)})^2}\\ 
    &\phantom{=} +\sum_{1 \leq c_2 \leq 2(p-b)} 2 \cdot
    \Indic{(-1)^{c_1+c_2}=(-1)^{p-b}} \cdot
    c_2\Prob{\mathcal{O}(\delta \psi^{-1})=c_2}\Big]\\
&= \sum_{1 \leq c_1 \leq 2(p-b)} c_1 \Prob{\mathcal{O}(\psi)=c_1} \Big[\BigO{(\ln{(p-b)})^2} + \BigO{\ln{(p-b)}}\Big]\\
&= \BigO{(\ln{(p-b)})^2} \sum_{1 \leq c_1 \leq 2(p-b)} c_1 \Prob{\mathcal{O}(\psi)=c_1}\\
&= \BigO{(\ln{(p-b)})^3}.
\end{align*}
This implies that
\begin{equation*}
\Expect{\left(\mathcal{O}(\psi) + \mathcal{O}(\delta \psi^{-1})\right)^2}= \BigO{(\ln{(p-b)})^3},
\end{equation*}
and so $\Expect{\mathcal{O}_{0,1}\mathcal{O}_{D,0}}=\BigO{(\ln{(p-b)})^3}$ as well.
\end{proof}

We can now prove \Cref{quartic-face-moments}.

\begin{proof}[of \Cref{quartic-face-moments}]
To prove \Cref{quartic-face-mean}, we first make a clear list of the different types of faces:
\begin{itemize}
\item the faces without color 0 are, by definition, subsets of the $\hat{0}$-bubbles. They either contain the distinguished color of their $\hat{0}$-bubble, in which case they have 4 vertices, or they do not, in which case they have 2 vertices. There are, in each $\hat{0}$-bubble, $D-1$ faces of the first type, and $(D-1)(D-2)$ of the second.
\item For each $i \neq 0$, one can see the bicolored graph $(Q^D_p)_{i,0}$ as $2p$ vertices with one in- and one out-half-edge, with a uniform matching of the half-edges, similarly to the construction of $S^{D,j}_p$. The faces of color $\{0,i\}$ are thus given by the cycles of a uniform permutation of size $2p$.
\end{itemize}
Therefore:
\begin{align*}
\Expect{b_2(Q^D_p)}&= \left( (D-1)(D-2) + (D-1) \right)p + D\left(\sum_{j=1}^{2p}\frac{1}{j}\right)\\
& = (D-1)^2p + D(\ln{(2p)} + \gamma) + \smallo{1}.
\end{align*}

To obtain the variance, we now have to calculate $\Expect{\left(b_2(Q^D_p)\right)^2}$. We decompose it into different parts as follows:
\begin{align*}
\Expect{\left(b_2(Q^D_p)\right)^2} &= \Expect{\left(\sum_{0 \leq i < j \leq D} \mathcal{O}\left(\alpha_i \alpha_j^{-1}\right)\right)^2}\\
&=\underbrace{\Expect{\left(\sum_{1 \leq i < j \leq D} \mathcal{O}\left(\alpha_i \alpha_j^{-1}\right)\right)^2}}_{E_1} + \underbrace{\ \ 2 \ \ \Expect{\sum_{\substack{1 \leq i \leq D \\ 1 \leq k < l \leq D }} \mathcal{O}\left(\alpha_0 \alpha_i^{-1}\right)\mathcal{O}\left(\alpha_k \alpha_l^{-1}\right)}}_{E_2}\\
&\phantom{=} + \underbrace{\sum_{1\leq i \leq D}\Expect{ \mathcal{O}\left(\alpha_0 \alpha_i^{-1}\right)^2}}_{E_3} + \ \ \underbrace{ \ \ 2 \sum_{1\leq i < j \leq D}\Expect{ \mathcal{O}\left(\alpha_0 \alpha_i^{-1}\right)\mathcal{O}(\alpha_0 \alpha_j^{-1})}}_{E_4}.
\end{align*}
In the first term, $E_1$, there are all the faces without color 0. By definition, the number of those is fixed:
\begin{equation*}
\sum_{1 \leq i < j \leq D} \mathcal{O}\left(\alpha_i \alpha_j^{-1}\right)=p((D-1)(D-2)+(D-1))=p(D-1)^2,
\end{equation*}
so that $E_1=p^2(D-1)^4$. Moving on to $E_2$, for any $i,k,l \in \{1,2, \dots, D\}$, the permutations $\alpha_0 \alpha_i^{-1}$ and $\alpha_k \alpha_l^{-1}$ are independent, so:
\begin{align*}
E_2&=2 \, \Expect{\sum_{1 \leq i \leq D} \mathcal{O}\left(\alpha_0 \alpha_i^{-1}\right)} \cdot \Expect{\sum_{1 \leq k < l \leq D} \mathcal{O}\left(\alpha_k \alpha_l^{-1}\right)} \\
&= 2 D\left(\sum_{j=1}^{2p}\frac{1}{j}\right) \cdot (D-1)^2 p,
\end{align*}
using, for the first part, the fact that all $i \in \{1, 2, \dots, D\}$, $\alpha_0 \alpha_i^{-1}$ is uniform, and, for the second part, the same argument as for $E_1$. Thus
\begin{equation*}
E_2=2 D(D-1)^2 p\left(\sum_{j=1}^{2p}\frac{1}{j}\right).
\end{equation*}
For $E_3$, we are dealing with uniform permutations, so:
\begin{align*}
E_3&= D((\ln{2p})^2+(2\gamma -1)\ln{(2p)}+\smallo{\ln{p}})\\
&= D(\ln{p})^2 + \BigO{\ln{p}}.
\end{align*}
We are now left with $E_4$, which we estimate thanks to \Cref{quartic-simpl-asympt}: $E_4=\BigO{(\ln{p})^3}$.

Therefore
\begin{align*}
\Var{b_2(Q^D_p)}&=\left(p^2 (D-1)^4 + 2 D(D-1)^2 p\left(\sum_{j=1}^{2p}\frac{1}{j}\right) +\BigO{(\ln{p})^3} \right)\\
 & \phantom{=} - \left((D-1)^2p + D\left(\sum_{j=1}^{2p}\frac{1}{j}\right)\right)^2\\
 & = \BigO{(\ln{p})^3}.
\end{align*}
\end{proof}

\begin{proof}[of \Cref{quartic-jacket-normal-law}]
Like in \Cref{subsection-unif-degree}, by symmetry, we assume without loss of generality that $\mathcal{J}_p$ corresponds to the usual cyclic ordering $(0 \, 1 \, \dots \, D)$, and in that case we are interested in the quantity:
\begin{equation*}
F = \sum_{0 \leq i  \leq D}\mathcal{O}_{i,i+1} 
\end{equation*} 
where, by convention, $\mathcal{O}_{D,D+1} = \mathcal{O}_{D,0}$. From the structure of the $\alpha_i$, for $1 \leq i \leq D$, we deduce
\begin{equation*}
\sum_{1 \leq i \leq D-1}\mathcal{O}_{i,i+1} + \mathcal{O}_{D,1} = 2p(D-1),
\end{equation*}
hence:
\begin{equation*}
F = \mathcal{O}_{0,1} + \mathcal{O}_{D,0} + 2p(D-1) - \mathcal{O}_{D,1}.
\end{equation*}

Thus, up to a constant, $F$ follows the same law as $f:= \mathcal{O}_{0,1} + \mathcal{O}_{D,0}  - \mathcal{O}_{D,1}$. Now, conditionally to the number $2b$ of fixed points of $\alpha_D \alpha_1^{-1}$, \Cref{quartic-simpl-asympt} gives us the behavior of $f$. Indeed, it has the same law as $\mathcal{O}(\psi) + \mathcal{O}(\delta \psi^{-1}) - (p+b)$, with the notations of \Cref{quartic-simpl-asympt}.\\

As explained above in the proof of \Cref{quartic-asympt-bound}, conditionally to the conjugacy class $\mathcal{C}_{\psi}$ of $\psi$, the law of $\delta \psi^{-1}$ converges for the total variation distance to the uniform measure either on $\mathcal{A}_{2(p-b)}$ (if $\mathcal{C}_{\psi}$ has the same parity as $\delta$), or on  $(\mathcal{A}_{2(p-b)})^c$ (if they are of opposite parities). 
Thus:
\begin{equation}\label{f-semi-convo}
\begin{aligned}
\Prob{f=l}&= \sum_{1 \leq b \leq p}\Prob{\mathcal{O}_{D,1}=p+b}\Prob{\mathcal{O}(\psi) + \mathcal{O}(\delta \psi^{-1})=l+p+b}\\
&=\sum_{1 \leq b \leq p}\Prob{\mathcal{O}_{D,1}=p+b} \Big[2 \cdot \Indic{l \text{ \scriptsize even}} \left(P_{\psi}* P_{\delta \psi ^{-1}}\right)(l+p+b) +&\\
  &&\hspace{-1cm}\BigO{\frac{(\ln{(p-b)})^2}{p-b}} \Big]
\end{aligned}
\end{equation}
where $P_{\sigma}$ is the law of the number of cycles of $\sigma$.\\

Now, let us recall that for a uniform permutation $\sigma \in \mathfrak{S}_n$, $\mathcal{O}(\sigma)$ has the same distribution as $\sum_{j=1}^{n}B_j$, where the $B_j$ are independent Bernoulli variables, with $B_j$ of parameter $1/j$. We use this fact to bound the dependence on $b$ of $P_{\psi}$ and $P_{\delta \psi^{-1}}$. Indeed, as we will later use Hoeffding's inequality for $b$, we consider the case $\lvert b - p \left(\frac{D-2}{D}\right) \rvert \leq \sqrt{p\ln{p}}$. We can couple $\mathcal{O}(\psi)$ with a variable of the form $ \sum_{j=1}^{\lfloor \frac{4p}{D} \rfloor} B_j$,by taking or adding some $B_j$ to the sum to get $\mathcal{O}(\psi)$, and in that case:
\begin{equation}
\label{psi-approx-1}
\Prob{\Big\lvert \mathcal{O}(\psi) - \sum_{j=1}^{\lfloor \frac{4p}{D} \rfloor} B_j \Big\rvert \geq 1} = \BigO{\sqrt{\frac{\ln{p}}{p}}}.
\end{equation}
Indeed:
\begin{equation*}
\Big\lvert \mathcal{O}(\psi) - \sum_{j=1}^{\lfloor \frac{4p}{D} \rfloor} B_j \Big\rvert \leq B\left(\BigO{\sqrt{p\ln{p}}},\Big\lfloor \frac{4p}{D} \Big\rfloor \right)
\end{equation*}
where $B(\BigO{\sqrt{p\ln{p}}},\lfloor \frac{4p}{D} \rfloor )$ follows a binomial distribution of parameters $(\BigO{\sqrt{p\ln{p}}},(\lfloor \frac{4p}{D} \rfloor)^{-1} )$. Then, applying Markov's inequality to $B(\BigO{\sqrt{p\ln{p}}},\lfloor \frac{4p}{D} \rfloor )$:
\begin{equation*}
\Prob{B\left(\BigO{\sqrt{p\ln{p}}}, \Big\lfloor \frac{4p}{D} \Big\rfloor \right) \geq 1 } = \smallo{\frac{1}{\sqrt{\ln{p}}}}.
\end{equation*}

Let us note $g^D_p$ the law of $\sum_{j=1}^{\lfloor \frac{4p}{D} \rfloor} B_j$. Similarly to what we did in the proof of \Cref{unif-jacket-normal-law}, we show that $g^D_p$ converges uniformly to a discrete gaussian distribution, of mean $E_p=\ln{p} + \gamma + \ln(\frac{4}{D}) + \smallo{1}$, and of variance  $V_p=\ln{p} + \gamma + \ln(\frac{4}{D}) - \pi^2/6 + \smallo{1}$:
\begin{equation}
\label{psi-approx-2}
g^D_p(c)=\frac{(1+\smallo{1})\exp{\left(\frac{-(c-E_p)^2}{2V_p}\right)}}{\sqrt{2\pi V_p}}.
\end{equation}
As \Cref{psi-approx-2} implies in particular that $g^D_p(c)$ is slowly varying, we deduce from \Cref{psi-approx-1}:
\begin{equation}
\label{psi-approx-3}
P_{\psi}(c)=g^D_p(c)+\smallo{\frac{1}{\sqrt{\ln{p}}}}.
\end{equation}
Now, applying Hoeffding's inequality to $b$, we get:
\begin{equation}
\label{b-hoeffding}
\Prob{\Big\lvert b - p\left(\frac{D-2}{D} \right)\Big\rvert \geq \sqrt{p\ln{p}}}= \BigO{\frac{1}{p^2}}.
\end{equation}
Combining \Cref{f-semi-convo,psi-approx-3,b-hoeffding}, we can write:
\begin{equation*}
\Prob{f=l}= 2\cdot \Indic{l \text{ \scriptsize even}} (\tilde{P}_{D,1} * g^D_p * g^D_p)(l) + \smallo{\frac{1}{\sqrt{\ln{p}}}}
\end{equation*}
where $\tilde{P}_{D,1}$ is the law of $-\mathcal{O}_{D,1}$.\\

Thus:
\begin{equation*}
\Big\lVert P_f - 2 P_h \cdot \Indic{h \text{ \scriptsize even}}\Big\rVert = \smallo{\frac{1}{\sqrt{\ln{p}}}},
\end{equation*}
where $h=\sum_{1 \leq j \leq \lfloor \frac{4p}{D} \rfloor}X_j + Y + p$, with the $X_j$ and $Y$ being independent binomial variables, with $X_j$ of parameters $(2, 1/j)$ and $Y$ of parameters $(p, (D-2)/D)$. The Lindeberg Central Limit Theorem and the Local Limit Theorem apply once again, and the resulting convergence of $P_h$ to a discrete gaussian implies that of $P_f$.\\
Finally, as $\Expect{f}=2p(D-1)/D+2\ln{p}+\BigO{1}=\Expect{h} +\BigO{1}$, and $\Var{f}=2p(D-2)/D^2+\BigO{(\ln{p})^3}=\Var{h}+\BigO{(\ln{p})^3}$, we have a weak convergence result for $(f- \Expect{f})/\sqrt{\Var{f}}$.
\end{proof}

As anounced before, the arguments we have used here to get results on the degree of $Q^{D}_p$ also imply a Central Limit Theorem for the genus of a uniform random map, in a sense that we make more precise now. Consider, for a given $p \geq 1$, a uniform fixed-point-free involution of $\{1, \dots, 2p\}$, $\delta$, and a uniform permutation $\psi \in \mathfrak{S}_{2p}$, independent from $\delta$. Then $\delta, \psi, \delta \psi^{-1}$ respectively describe the edges, faces and vertices of a (half-edge-labelled) ribbon graph $M_p$ with $p$ edges, which is clearly uniform on the set of such graphs.\\

We first prove that $M_p$ is \emph{a.a.s.} connected:

\begin{prop}
One has:
\begin{equation}
\label{rib-graph-conn}
\Prob{M_p \text{ connected}}= 1 - \frac{1}{2p-1} + \BigO{\frac{1}{p}}.
\end{equation}
\end{prop}

\begin{proof}
Notice that the probability of $M_p$ being connected is exactly the same as the probability of $Q^{D}_p$ being connected (for any $D \geq 2$). Indeed, in both cases we consider the set of indices $\{1, \dots, 2p\}$ grouped into pairs (by $\delta$ for $M_p$ and by the structure of the $\hat{0}$-bubbles for $Q^{D}_p$), and then add connections according to a uniform permutation independent from the pairing ($\psi$ for $M_p$ and $\alpha_{0}$ for $Q^D_p$). We therefore deduce \Cref{rib-graph-conn} directly from \Cref{quartic-connect}.
\end{proof}

Now that we know that $M_p$ is connected \emph{a.a.s.}, we can consider its genus, for which we prove the following theorem:
\begin{thm}
\label{rib-graph-jack}
Let $g_p$ be the genus of $M_p$. Then the quantity $\frac{g_p - \Expect{g_p}}{\sqrt{\Var{g_p}}}$ converges weakly to the standard normal distribution. Moreover, we have the following estimations:
\begin{align*}
\Expect{g_p}&=\frac{p}{2} - (\ln{2p} + \gamma) +1 + \smallo{1},\\
\Var{g_p}&=\BigO{(\ln{p})^3}.
\end{align*}
\end{thm}

\begin{proof}
Let us start with the first statement of the theorem. As explained in the proof of \Cref{quartic-asympt-bound}, conditionally to the conjugacy class $\mathcal{C}_{\psi}$ of $\psi$, $\delta \psi^{-1}$ is asymptotically uniform on $H=\mathcal{A}_{2p}$ (resp. $H=\left(\mathcal{A}_{2p}\right)^c$) if $\psi$ and $\delta$ are of the same parity (resp. of opposite parities), indeed:
\begin{equation*}
\lVert P_{\delta \psi^{-1}} - \mathcal{U}_H \rVert=\BigO{\frac{(\ln{p})^2}{p}}.
\end{equation*} 
Therefore, just like we derived \Cref{unif-jacket-faces-law} from \Cref{unif-asympt-indep}, we get:
\begin{equation*}
\lVert P_{\delta, \psi} - 2 \cdot \Indic{C_{\delta, \psi}+p \text{ even}} P_{\psi} * P_{\delta \psi^{-1}} \rVert = \BigO{\frac{(\ln{p})^2}{p}}
\end{equation*}
where $P_{\delta,\psi}$ is the law of $C_{\delta, \psi}:=\mathcal{O}(\psi) + \mathcal{O}(\delta \psi^{-1})$.

Now, following the same arguments as in the proof of \Cref{unif-jacket-normal-law}, we obtain that the quantity $\frac{C_{\delta, \psi} - \Expect{C_{\delta, \psi}}}{\sqrt{C_{\delta, \psi}}}$ converges weakly to the standard normal law. Finally, as $g_p=\frac{1}{2}\left(p-C_{\delta, \psi}\right)+1$, we have a weak convergence result for $g_p$ as well. As for the estimations of $\Expect{g_p}$ and $\Var{g_p}$, the first is a direct consequence of the fact that $\psi$ and $\delta \psi^{-1}$ are uniform permutations, and the second was proved in the proof of \Cref{quartic-asympt-bound}.
\end{proof}

\section{Uniform-uncolored models}
\label{sect-unif-decol}

Random $D$-complexes studied by physicists are mainly of the type of the quartic model: the $\hat 0$-bubbles of the corresponding colored graphs all belong to a fixed finite set of $D$-colored graphs. In the quartic model studied in \cref{sect-quartic}, this set is reduced to a singleton, namely the quartic melonic graph of \Cref{quartic-mel}. Melonic $D+1$-colored graphs are dual to special colored triangulations of the $D$-sphere \cite{gurau-ryan}. A natural question is the following: do the results obtained for the quartic model also hold in the case of more generic $\hat 0$-bubbles set and particularly if this set contains non melonic graphs? To answer this question, we now consider a type of model similar to the quartic one, where the $\hat{0}$-bubbles are all identical, up to color permutations.\\

More precisely, we start from a fixed, connected bipartite $D$-colored graph $G$ with $2t \geq 4$ vertices, and we consider a random $(D+1)$-colored graph $G_p$, with $p$ $\hat{0}$-bubbles $\mathcal{B}^{\hat{0}}_1, \dots, \mathcal{B}^{\hat{0}}_p$, which are all copies of $G$ up to a color permutation: if an edge $e$ of $G$ has color $c(e) \in \{1, \dots, D\}$, then the corresponding edge $e'$ of  $\mathcal{B}^{\hat{0}}_k$ will have color $c_k(e')=\gamma_k(c(e))$, where $\gamma_1, \dots, \gamma_p$ are i.i.d. uniform permutations in $\mathfrak{S}_D$ (see \Cref{unif-decol-model} for an illustration).\\
Similarly to the quartic case, the $0$-colored edges of $G_p$ are given by a permutation $\alpha_0$ uniform on $\mathfrak{S}_{p\cdot t}$, and independent from the rest (we can choose an arbitrary labelling $v^{B}_1, \dots,v^{B}_t, v^{W}_1, \dots, v^{W}_t$ of the vertices of $G$ to obtain a canonical labelling of the vertices of $G_p$).\\

The connectivity properties of this model are very similar to the specific case of the quartic model seen in \Cref{subsect-quartic-cc}, as detailed in \Cref{subsect-unif-decol-cc}. In particular, the complex of the general uniform-uncolored model $\Delta(G_p)$ exhibits the same asymptotic structure as for the quartic model, \emph{i.e.} a majority of $0$-points, most of which having the ``colored hubs'' as common neighbors, as explained in \Cref{subsect-unif-decol-geom}. As for the behaviors of the degree and of the genera of the jackets, they are difficult to obtain in the general case, as we do not have an explicit description of $G_p$ in terms of $D+1$ random permutations $(\alpha_0, \alpha_1, \dots, \alpha_D) \in \mathfrak{S}_p^{D+1}$. Even if we use such a description for a particular case, the laws of the $\alpha_i$ will typically be much more complicated than for $Q^D_p$, and might not yield explicit results or estimations as easily. Moreover, because of the unsatisfying connectivity properties of this model, it does not seem very worthwhile to undertake such a laborious task.

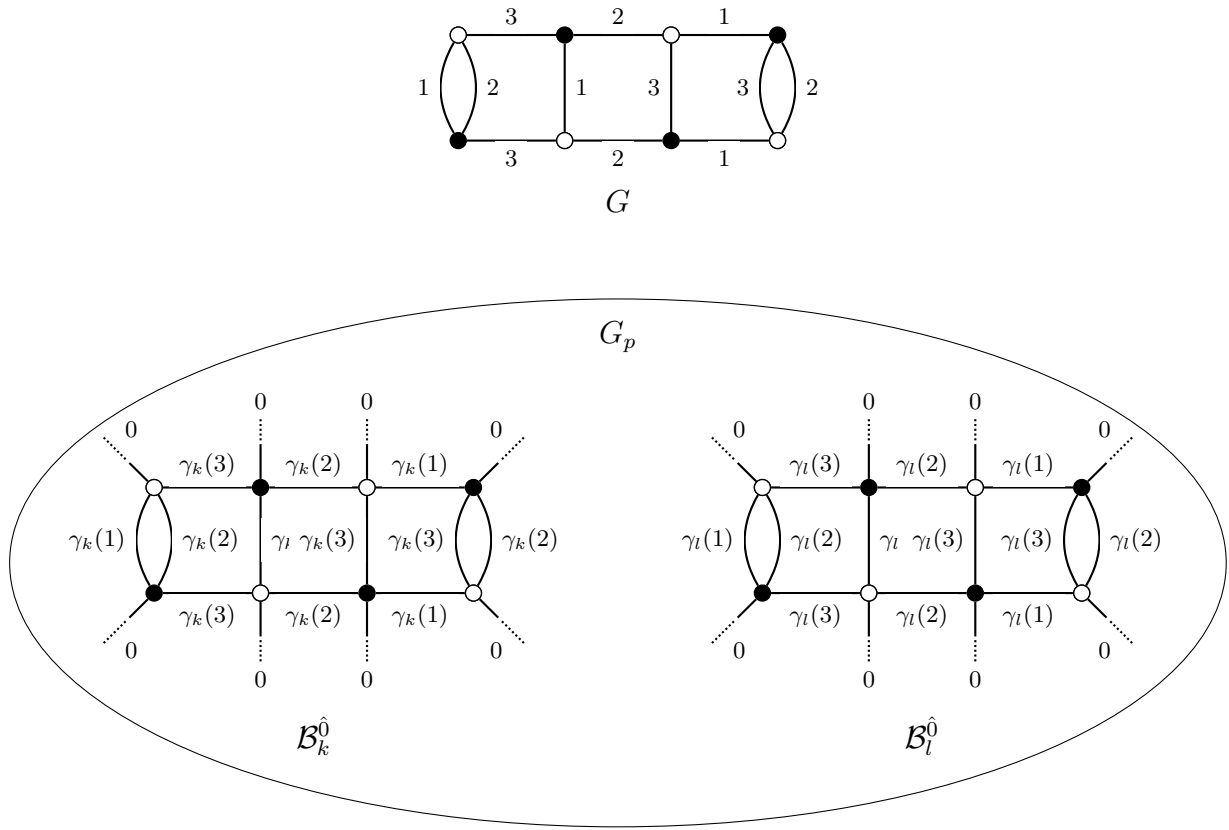
\begin{figure}[htp]
\begin{tikzpicture}
\SetGraphUnit{1.4}
\GraphInit[vstyle=simple]
\Vertex[x=0, y=0]{1u1} 
\Vertex[x=-4, y=-6]{1u3}
\Vertex[x=4, y=-6]{1u4}
\foreach \s in {1,3,4}{
\SO(1u\s){1d\s};
}
\foreach \i in {1,2,3}{
  \pgfmathtruncatemacro{\j}{\i+1};
  \pgfmathtruncatemacro{\c}{4-\i};
  \foreach \s in {1,3,4}{
    \EA(\i u\s){\j u\s};
    \EA(\i d\s){\j d\s};
    \ifnum \s=1
      \def\t{}; \def\u{};
    \else 
      \ifnum \s=3
        \def\t{$\gamma_k ($}; \def\u{)};
      \else
        \def\t{$\gamma_l ($}; \def\u{)};
      \fi
    \fi
    \Edge[labelstyle={above}, label={\footnotesize \t\c\u}](\i u\s)(\j u\s); 
    \Edge[labelstyle={below}, label={\footnotesize \t\c\u}](\i d\s)(\j d\s);
    \ifnum \i=2
      \Edge[labelstyle={right},label={\footnotesize \t 1\u}](\i u\s)(\i d\s);
    \fi
    \ifnum \i=3
      \Edge[labelstyle={left},label={\footnotesize \t\i\u}](\i u\s)(\i d\s);
      \Edge[style={bend right}, labelstyle={left}, label={\footnotesize \t 1\u}](1u\s)(1d\s); 
      \Edge[style={bend left}, labelstyle={right}, label={\footnotesize \t 2\u}](1u\s)(1d\s);
      \Edge[style={bend right}, labelstyle={left}, label={\footnotesize \t 3\u}](4u\s)(4d\s); 
      \Edge[style={bend left}, labelstyle={right}, label={\footnotesize \t 2\u}](4u\s)(4d\s);
    \fi
  }
}
\foreach \z in {0,1}{
\ifnum \z=0
  \def\symb{}; \def\L{}; \def\st{};
  \else \def\symb{z}; \def\L{\footnotesize 0}; \def\st{densely dotted};
  \fi
  \foreach \s in {3,4}{
    \NOWE[unit=0.35, empty=true](1u\s\symb){1u\s\symb z};
    \Edge[style={\st}, labelstyle={above right}, label=\L](1u\s\symb)(1u\s\symb z);
    \SOWE[unit=0.35, empty=true](1d\s\symb){1d\s\symb z};
    \Edge[style={\st}, labelstyle={below right}, label=\L](1d\s\symb)(1d\s\symb z);
    \NOEA[unit=0.35, empty=true](4u\s\symb){4u\s\symb z};
    \Edge[style={\st}, labelstyle={above left}, label=\L](4u\s\symb)(4u\s\symb z);
    \SOEA[unit=0.35, empty=true](4d\s\symb){4d\s\symb z};
    \Edge[style={\st}, labelstyle={below left}, label=\L](4d\s\symb)(4d\s\symb z);
    \foreach \i in {2,3}{
      \NO[unit=0.6, empty=true](\i u\s\symb){\i u\s\symb z};
      \SO[unit=0.6, empty=true](\i d\s\symb){\i d\s\symb z};
      \Edge[style={\st}, labelstyle={above}, label=\L](\i u\s\symb)(\i u\s\symb z);
      \Edge[style={\st}, labelstyle={below}, label=\L](\i d\s\symb)(\i d\s\symb z);
    }
  }
}
\foreach \s in {1,3,4}{
\AddVertexColor{white}{1u\s,3u\s,2d\s,4d\s};}
\draw (2.1,-7) ellipse (8 and 3.5);
\node at (2.1,-2.2) {\large$G$};
\node at (2.1,-4) {\large$G_p$};
\node at (-1.9,-9.3) {\large$\mathcal{B}^{\hat{0}}_k$};
\node at (6.1,-9.3) {\large$\mathcal{B}^{\hat{0}}_l$};
\end{tikzpicture}
\caption{Starting from a $D$-colored graph $G$ (here $D=3$), we construct a $(D+1)$-colored graph $G_p$, whose $p$ $\hat{0}$-bubbles are copies of $G$ in which the colors have been permuted.}
\label{unif-decol-model}
\end{figure}

\subsection{Connectedness}
\label{subsect-unif-decol-cc}

Just like in the quartic case, we show that $G_p$ is connected a.a.s., and moreover that the expectation value of its number of components converges to 1:
\begin{thm}
One has:
\begin{align*}
\Prob{G_p \text{ connected}}&=1 - \frac{p}{\begin{pmatrix}tp \\t \end{pmatrix}} + \BigO{\frac{1}{p^{2(t-1)}}},\\
\Expect{k(G_p)}&= 1 + \BigO{\frac{1}{p^{t-1}}}.
\end{align*}
\end{thm}
\begin{proof}
  We proceed in the same way as for \Cref{quartic-connect,quartic-cc}, using the fact that, for $1\leq k \leq \lfloor \frac{p}{2} \rfloor$, $\Prob{G_p \text{ has a closed subgraph containing $k$ copies of $G$}}=\binom{p}{k}/\binom{tp}{tk}$.
 \end{proof}

\begin{figure}[htp]
\centering
\begin{tikzpicture}
\SetGraphUnit{1.5}
\GraphInit[vstyle=simple]
\Vertex[x=0, y=0]{1u1} 
\Vertex[x=0, y=-4]{1u2}
\Vertex[x=0, y=-8]{1u3}
\foreach \s in {1,2,3}{
\SO(1u\s){1d\s};
}
\foreach \i in {1,2,3}{
  \pgfmathtruncatemacro{\j}{\i+1};
  \pgfmathtruncatemacro{\c}{4-\i};
  \foreach \s in {1,2,3}{
    \EA(\i u\s){\j u\s};
    \EA(\i d\s){\j d\s};
    \ifnum \s=\c
      \def\L{\footnotesize $i$}; \SetUpEdge[color=gray];
    \else \def\L{}; \SetUpEdge[color=black];
    \fi
      \Edge[labelstyle={above}, labeltext={gray}, label=\L](\i u\s)(\j u\s);
      \Edge[labelstyle={below}, labeltext={gray}, label=\L](\i d\s)(\j d\s);
    \ifnum \i=2
      \ifnum \s=1
        \def\L{\footnotesize $i$};\SetUpEdge[color=gray];
      \else \def\L{}; \SetUpEdge[color=black];
      \fi
      \Edge[labelstyle={right}, labeltext={gray}, label=\L](\i u\s)(\i d\s);
    \fi
    \ifnum \i=3
      \ifnum \s=3
        \def\L{\footnotesize $i$}; \SetUpEdge[color=gray];
      \else \def\L{}; \SetUpEdge[color=black];
      \fi
        \Edge[labelstyle={right}, labeltext={gray}, label=\L](\i u\s)(\i d\s);   
      
      \ifnum \s=1
        \def\M{\footnotesize $i$}; \SetUpEdge[color=gray];
      \else \def\M{}; \SetUpEdge[color=black];
      \fi
      \Edge[style={bend right}, labelstyle={left}, labeltext={gray}, label=\M](1u\s)(1d\s);
      \ifnum \s=2
        \def\N{\footnotesize $i$}; \SetUpEdge[color=gray];
      \else \def\N{}; \SetUpEdge[color=black];
      \fi
      \Edge[style={bend left}, labelstyle={right}, labeltext={gray}, label=\N](1u\s)(1d\s);
      \Edge[style={bend left}, labelstyle={right}, labeltext={gray}, label=\N](4u\s)(4d\s);
      \ifnum \s=3
        \def\L{\footnotesize $i$}; \SetUpEdge[color=gray];
      \else \def\L{}; \SetUpEdge[color=black];
      \fi
      \Edge[style={bend right}, labelstyle={left}, labeltext={gray}, label=\L](4u\s)(4d\s); 
    \fi
  }
}
\foreach \s in {1,2,3}{
\AddVertexColor{white}{1u\s,3u\s,2d\s,4d\s};}
\draw[->] (6.75,-0.75) -- (8,-0.75);
\draw[->] (6.75,-4.75) -- (8,-4.75);
\draw[->] (6.75,-8.75) -- (8,-8.75);

\tikzset{VertexStyle/.append style = {minimum size = 3pt, inner sep = 0pt}}
\Vertex[x=10,y=-0.75]{s11}; 
\Vertex[x=12.5,y=-0.75]{s12};
\foreach \p [count=\i from 2] in {\NOWE, \NO, \NOEA, \SOEA, \SO, \SOWE}{
  \foreach \z in {0,1}{
    \ifnum \z=0
      \def\symb{}; \def\u{0.6};
      \pgfmathtruncatemacro{\j}{or(\i==2, or(\i==4, \i==6))};
      \ifnum \j=1
        \def\st{-<-=.5};
      \else \def\st{->-=.5};
      \fi
    \else \def\symb{\i}; \def\st{densely dotted}; \def\u{0.3};
    \fi
    \p[unit=\u,empty=true](s11\symb){s11\symb\i};
    \Edge[style={\st}](s11\symb)(s11\symb\i); 
    \ifnum \i=2
      \SO[unit=\u,empty=true](s12\symb){s12\symb\i};
      \Edge[style={\st}](s12\symb)(s12\symb\i);
    \fi
    \ifnum \i=3
      \NO[unit=\u,empty=true](s12\symb){s12\symb\i};
      \Edge[style={\st}](s12\symb)(s12\symb\i);
    \fi
  }
}

\Vertex[x=10,y=-4.75]{s21}; 
\Vertex[x=12.5,y=-4.75]{s22};
\foreach \p in {1,2}{
  \foreach \q [count=\i] in {\NOWE,\NOEA, \SOEA, \SOWE}{
    \foreach \z in {0,1}{
       \ifnum \z=0
         \def\symb{}; \def\u{0.6};
         \pgfmathtruncatemacro{\j}{or(\i==1, \i==3)};
         \ifnum \j=1
           \def\st{-<-=.5};
         \else \def\st{->-=.5};
         \fi
       \else \def\symb{\i}; \def\st{densely dotted}; \def\u{0.3};
       \fi
       \q[unit=\u,empty=true](s2\p\symb){s2\p\symb\i};
       \Edge[style={\st}](s2\p\symb)(s2\p\symb\i); 
    }
  }
}

\Vertex[x=12.5,y=-8.75]{s31}; 
\Vertex[x=10,y=-8.75]{s32};
\foreach \p [count=\i from 2] in {\NOWE, \NO, \NOEA, \SOEA, \SO, \SOWE}{
  \foreach \z in {0,1}{
    \ifnum \z=0
      \def\symb{}; \def\u{0.6};
      \pgfmathtruncatemacro{\j}{or(\i==2, or(\i==4, \i==6))};
      \ifnum \j=1
        \def\st{->-=.5};
      \else \def\st{-<-=.5};
      \fi
    \else \def\symb{\i}; \def\st{densely dotted}; \def\u{0.3};
    \fi
    \p[unit=\u,empty=true](s31\symb){s31\symb\i};
    \Edge[style={\st}](s31\symb)(s31\symb\i); 
    \ifnum \i=2
      \SO[unit=\u,empty=true](s32\symb){s32\symb\i};
      \Edge[style={\st}](s32\symb)(s32\symb\i);
    \fi
    \ifnum \i=3
      \NO[unit=\u,empty=true](s32\symb){s32\symb\i};
      \Edge[style={\st}](s32\symb)(s32\symb\i);
    \fi
  }
}
\end{tikzpicture}
\caption{For the purpose of studying the $\hat{\imath}$-bubbles of $G_p$, we replace each of its $\hat{0}$-bubbles by vertices, whose number, in- and out-degrees are prescribed by the positions of the $i$-edges in the $\hat{0}$-bubble. In this example we have taken the same base graph $G$ as in \Cref{unif-decol-model}.}
\label{unif-decol-simpl}
\end{figure}

Now, to obtain results about the asymptotic number of $\hat{\imath}$-bubbles of $G_p$, for a fixed $i \in \{1, 2, \dots, D\}$, we will consider the simplified graph $S^{i}_p$, defined like in the quartic case: the deletion of the edges of color $i$ decomposes any $\hat{0}$-bubble $\mathcal{B}_k$ into a certain number $r_i$ of connected components $\mathcal{B}^{(1)}_k, \dots, \mathcal{B}^{(r_i)}_k$, having respectively, say, $2n_1, \dots, 2n_{r_i}$ vertices (with $\sum_k n_k=t$). We associate to each component $\mathcal{B}^{(s)}_k$ a vertex $v_{k,s}$ with $n_s$ out- (resp. in-)going half-edges, corresponding to its black (resp. white) vertices, so that each half-edge inherits the labelling of its corresponding vertex (see \Cref{unif-decol-simpl} for an example). We then consider the graph $S^{i}_p$ on the vertices $(v_{k,s})$ (corresponding to all the components of all the $\hat{0}$-bubbles of $(G_p)_{\hat{\imath}}$), obtained by taking a uniform matching of the in- and out-half-edges: as explained in \Cref{subsect-quartic-cc} for the quartic case, this translates the uniformity of $\alpha_0$, and thus the connected components of $(G_p)_{\hat{\imath}}$ correspond to those of $S^{i}_p$.\\

Now, just like in the quartic case, we want to apply \Cref{cooper-frieze-thm} to prove that $S^i_p$ has a giant component. First, let us translate the characteristics of $S^i_p$ in the notation of the directed configuration model. If deleting the color $j$ in $G$ gives rise to $r_j$ vertices of respective degrees $(\delta^j_1, \delta^j_1), \dots, (\delta^j_r, \delta^j_r)$ in our simplified formulation, then $S^{i}_p$ will have $q=\sum_{1 \leq j \leq D}q_j r_j$ vertices, where $q_j=\#\{k \in \{1, 2, \dots, p\} \big\lvert \gamma_k(j)=i\}$. And for each possible degree $1 \leq \delta \leq t$, the number of vertices of degree $(\delta, \delta)$ in $S^{i}_p$ will be: $l_{\delta,\delta}=\sum_{1 \leq j \leq D}q_jN_{\delta,j} $, with $N_{\delta,j}= \#\{v \in \{1, \dots, r_j\} \big\lvert \delta^j_v=\delta\}$.\\

Now, note that, as the $\gamma_k$ are uniform, a given $i \in \{1, 2, \dots, D\}$ has an equal probability to replace any color $j \in \{1,2, \dots, D\}$ in a given $\hat{0}$-bubble. Moreover, the $\gamma_k$ are independent, so that, for given $i, j \in \{1, 2, \dots, D\}$, the number of $\hat{0}$-bubbles in which $i$ replaces $j$ is a binomial variable of parameters $(\frac{1}{D},p)$:
\begin{equation*}
\#\{k \in \{1, 2, \dots, p\} \big\lvert \gamma_k(j)=i\}\sim\text{Bin}\left(\frac{1}{D},p\right).
\end{equation*}
Thus, applying Hoeffding's inequality, we have once again:
\begin{equation*}
\Prob{\Big\lvert\frac{q_j}{p} - \frac{1}{D} \Big\rvert \geq \sqrt{\frac{\ln{p}}{p}}} \leq \frac{2}{p^2},
\end{equation*}
and for $q_j \sim p/D$, we can write $l_{\delta,\delta}=p \cdot c_{\delta}(1+\smallo{1}), \, q=p \cdot c_q(1+\smallo{1}), \, \theta=\theta_0(1+\smallo{1}), \, d=d_0(1+\smallo{1})$, where:
\begin{alignat*}{2}
c_\delta&=\frac{1}{D}\sum_{1 \leq j \leq D}N_{\delta,j}, &\qquad c_q&=\sum_{1 \leq \delta \leq t}c_{\delta}\\
\theta_0&=\frac{\sum_{\delta} \delta c_{\delta}}{\sum_{\delta}c_{\delta}}, & d_0&=\frac{\sum_{\delta} {\delta}^2 c_{\delta}}{\sum_{\delta} {\delta}c_{\delta}}.
\end{alignat*}
\begin{figure}[htp]
\centering
\begin{tikzpicture}
\SetGraphUnit{3.5}
\GraphInit[vstyle=simple]
\begin{scope}[rotate=-72]   
\Vertices{circle}{v1,v2,v3,v4,v5,v6,v7,v8,v9,v10};
\end{scope}
\begin{scope}[rotate=-54] 
\Vertices[empty=true]{circle}{u1,u2,u3,u4,u5};
\end{scope}

\foreach \i in {1,2,...,5}{
  \tkzActivOff
  \coordinate (u\i1) at ($(u\i)!0.1!(0,0)$){};
  \tkzActivOn
  \Vertex[empty=true][Node]{u\i1}
  \Edge[style={densely dotted, thick}](u\i)(u\i1);
}

\foreach \dir [count=\c] in {right,above right,above left,left,flat}{
  \pgfmathtruncatemacro{\i}{2*\c};
  \pgfmathtruncatemacro{\j}{\i-1};
  \Edge[style={bend right}](v\i)(v\j);
  \Edge[style={bend left}](v\i)(v\j);
  \ifnum \i=10    
    \tkzActivOff
    \coordinate (v11) at ($0.75*(v10) + 0.25*(v1)$) {};
    \coordinate (v12) at ($0.25*(v10) + 0.75*(v1)$) {};
    \tkzActivOn
    \Vertex[empty=true][Node]{v11}
    \Vertex[empty=true][Node]{v12}
    \Edge[labelstyle={below},label={$j$}](v10)(v11);
    \Edge[labelstyle={below},label={$j$}](v12)(v1);
    \Edge[style={densely dotted}](v11)(v12);
  \else  
    \pgfmathtruncatemacro{\t}{or(\i==6, \i==8)};
    \pgfmathtruncatemacro{\j}{\i+1};  
    \def\lst{\dir=2.5pt};
    \Edge[style={bend right}, labelstyle={\lst}, label={$j$}](v\i)(v\j);
  \fi
  \AddVertexColor{white}{v\i};
}
\end{tikzpicture}
\caption{A pearl necklace of color $j$.}
\label{necklace}
\end{figure}
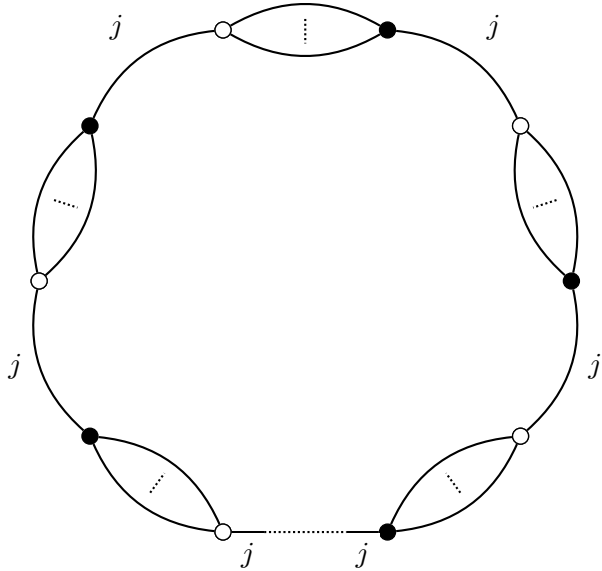

Thus, $S^i_p$ satisfies the hypotheses of \Cref{cooper-frieze-thm} as long as $d_0 > 1$, \emph{i.e.} as long as $c_1 < \sum_{\delta}c_{\delta}$. In other words, we can apply \Cref{cooper-frieze-thm} if there is at least one color $j \in \{1, 2, \dots, D\}$ such that $G_{\hat{\jmath}}$ is not made of melons only, which is always the case when $D \geq 2$. Indeed, if there is one color $j$ such that $G_{\hat{\jmath}}$ is only made of melons, then $G$ is necessarily a ``pearl necklace'' (see \Cref{necklace}), and in that case, for any other color $k \in \{1, 2, \dots, D \}\setminus\set{j}$, $G_{\hat{k}}$ is connected, and in particular not made of melons only.

\Cref{cycle-number-config-model} also applies, as we have, in the notation of the proposition: $\frac{l_{1,1}}{n}=\frac{c_1}{c_q} + \smallo{1}$. This yields the following theorem, analogous to \Cref{quartic-i-bubbles}:

\begin{thm}
\label{unif-decol-i-bubbles}
For $D \geq 3$, for any $i \in \{1, 2, \dots, D\}$, $G_p$ has a giant $\hat{\imath}$-bubble containing $2tp - \BigO{\sqrt{p \ln{p}}}$ vertices. Moreover, the expectation value of the number of $\hat{\imath}$-bubbles of $G_p$ is:
\begin{equation*}
\Expect{k(S^{i}_p)}= c_G +1 +o(1)
\end{equation*}
where $c_G=\sum_{k \geq 2}\frac{c_1^k}{k\left(c_q \theta_0\right)^k} < \infty$ as $c_1 < c_q \theta_0$.
\end{thm}

Hence, for the total number of $D$-bubbles of $G_p$:
\begin{cor}
\label{unif-decol-bubbles}
For $D \geq 3$, $\Expect{b_D(G_p)}= p + D \left(c_G +1 \right)+ \smallo{1}$.
\end{cor}

\subsection{Geometry of the complex}
\label{subsect-unif-decol-geom}

Just like in the particular case of the quartic model, the previous study of the $D$-bubbles of $G_p$ gives us insight on the typical geometry of the associated complex $\Delta(G_p)$. Indeed, once again there are $p$ 0-points in $\Delta(G_p)$, while for any $i \in \{1,2, \dots, D\}$ there are at most $\BigO{\sqrt{p \ln{p}}}$ $i$-points, one of which is the ``hub'' $i$-point having most 0-points as neighbors. The result of \Cref{quartic-avg-dist} thus also holds for any uniform-uncolored model, as long as $D \geq 3$:

\begin{thm}
Let $u,v$ be two vertices of $\Delta(G_p)$ chosen uniformly at random and independently. Then, if $D \geq 3$, $d(u,v)=2 \text{ a.a.s.}$.
\end{thm}

\begin{rem}
Let us note that what we have done for one fixed, connected base graph $G$ can be generalized further to models where this base graph $G$ is not necessarily connected (in which case, the copies of $G$ are not $\hat{0}$-bubbles but \emph{sets} of $\hat{0}$-bubbles), or where each $\hat{0}$-bubble is uniformly drawn from a finite set of (finite) $D$-colored graphs (with, once again, a randomization of the different colors): similar results will still apply, as long as the criterion $d_0 > 1$ is verified for the corresponding oriented configuration model. This is typically the case for all ``interaction vertices'' appearing in colored tensor models generalizing the quartic one.
\end{rem}

This brings an end to the pursuit of continuum scaling limits in the spirit of the Brownian sphere in uniform-uncolored models. As noted in \Cref{subsection-unif-conn}, the results of numerical simulations for Euclidean Dynamical Triangulations had already hinted that in higher dimensions, models of random triangulations with no or low constraints on curvature yield quite degenerate limit spaces. We can thus see both our uniform and uniform-uncolored models as instances of a crumpled phase, which would be the first ones to be investigated mathematically and not by simulations.

To find promising scaling limits, we must turn to more complicated models, and works in theoretical physics suggest that the Gurau degree (which is a way to quantify the curvature of trisps) must play a central role in our tentative distributions.

\section{Conclusion}

We have studied in this work two random models on bipartite $(D+1)$-colored graphs and the associated complexes. The first one, $U^D_p$, is uniform on the bipartite $(D+1)$-colored graphs with $2p$ labelled vertices. We have proved that, \emph{a.a.s.}, $U^D_p$ is connected and the associated complex $\Delta(U^D_p)$ has exactly one point of each color, which is a quite singular behavior, and is not satisfactory for the purpose of finding scaling limits. We have also obtained a Central Limit Theorem for the genus of one jacket of $U^D_p$. In the second one, $G_p$, we have fixed the $\hat{0}$-bubbles to be copies of a given $D$-colored graph $G$, and kept the matching of the 0-half-edges uniform. $G_p$ is also connected \emph{a.a.s.}, but the number of points of $\Delta(G_p)$ grows linearly with $p$. However, the average distance between two points of $\Delta(G_p)$ is \emph{a.a.s.} 2, which also halts our quest for a continuum limit. We have more extensively studied the special case of $Q^D_p$, in which $G$ is quartic. In that particular case,  we have also obtained a Central Limit Theorem for the genus of one jacket, and the arguments we employed also yielded a Central Limit Theorem for the genus of a uniform map of size $p$, as $p \to \infty$.\\

In the context of quantum gravity, we wish to find random models on colored complexes exhibiting a scaling limit, that can be interpreted as a continuum space-time. We have seen that these models do not fit this purpose. However, their study involved the adaptation of several combinatorial and probabilistic tools to the subject of edge-colored graphs, which will surely prove to be valuable when tackling more complicated models, as we plan to do. Those new models will undoubtely have to involve the Gurau degree more deeply, as it is a crucial quantity in the theory of colored tensor models.

\bigskip
\paragraph{Acknowledgements}

I warmly thank Fabien Vignes-Tourneret for his help and advice.

\contactrule
\contactACarrance

\newpage
\printbibliography

\end{document}